\documentclass[a4paper,11pt]{article}
\usepackage[T1]{fontenc}
\usepackage{amsmath,amsfonts,amsthm,amssymb}
\usepackage{authblk}
\usepackage{booktabs}
\usepackage{comment}
\usepackage{filecontents}
\usepackage[hmargin={28mm,28mm},vmargin={30mm,35mm}]{geometry}
\usepackage[latin1]{inputenc}
\usepackage{newtxtext,newtxmath}
\usepackage{paralist}
\usepackage{pgfplots}
\usepackage[colorlinks,allcolors={blue}]{hyperref}
\usepackage{enumerate}
\usepackage{multirow}
\usepackage{cancel}
\usepackage{xcolor}
\usepackage{graphicx}
\usepackage{subcaption}

\newcommand{\email}[1]{\href{mailto:#1}{#1}}

\newtheorem{theorem}{Theorem}
\newtheorem{proposition}[theorem]{Proposition}
\newtheorem{lemma}[theorem]{Lemma}
\newtheorem{corollary}[theorem]{Corollary}
\theoremstyle{remark}
\newtheorem{remark}[theorem]{Remark}
\theoremstyle{definition}

\renewcommand{\vec}[1]{\boldsymbol{#1}}    
\newcommand{\tens}[1]{\boldsymbol{#1}}
\newcommand{\trans}{^\top}

\DeclareMathOperator{\tr}{tr}

\DeclareMathOperator{\card}{card}

\newcommand{\SCAL}{{\cdot}}
\newcommand{\SSCAL}{{:}}
\newcommand{\GRAD}{\vec{\nabla}}
\newcommand{\GRADh}{\vec{\nabla}_h}
\newcommand{\DIV}{\vec{\nabla}{\cdot}}

\newcommand{\GRADs}{\tens{\nabla}_{\rm s}}
\newcommand{\GRADsh}{\tens{\nabla}_{{\rm s},h}}

\newcommand{\st}{\; : \;}

\newcommand{\Id}[1][d]{\tens{I}_{#1}}

\newcommand{\norm}[2][]{\|#2\|_{#1}}
\newcommand{\seminorm}[2][]{|#2|_{#1}}
\newcommand{\vvvert}{\vert\kern-0.25ex\vert\kern-0.25ex\vert}
\newcommand{\tnorm}[2][]{\vvvert #2\vvvert_{#1}}

\newcommand{\meas}[1]{|#1|}
\newcommand{\term}{\mathfrak{T}}

\newcommand{\Real}{\mathbb{R}}

\newcommand{\Poly}[2][]{\mathbb{P}_{#1}^{#2}}

\newcommand{\Mh}[1][h]{\mathcal{M}_{#1}}
\newcommand{\Th}[1][h]{\mathcal{T}_{#1}}
\newcommand{\Fh}[1][h]{\mathcal{F}_{#1}}
\newcommand{\Fhi}{\Fh^{{\rm i}}}
\newcommand{\Fhb}{\Fh^{{\rm b}}}

\newcommand{\fTh}[1][h]{\mathfrak{T}_{#1}}
\newcommand{\fFh}[1][h]{\mathfrak{F}_{#1}}

\newcommand{\normal}{\vec{n}}

\newcommand{\lproj}[2][h]{\pi_{#1}^{#2}}
\newcommand{\vlproj}[2][h]{\vec{\pi}_{#1}^{#2}}
\newcommand{\tlproj}[2][h]{\tens{\pi}_{#1}^{#2}}
\newcommand{\eproj}[2][h]{\varpi_{#1}^{#2}}
\newcommand{\veproj}[2][h]{\vec{\varpi}_{#1}^{#2}}

\newcommand{\Iav}[1][1]{I_{{\rm av},h}^{#1}}
\newcommand{\vIav}[1][1]{\vec{I}_{{\rm av},h}^{#1}}

\newcommand{\Cerr}[2]{\mathcal{E}_{h}({#1};{#2})}

\newcommand{\uvec}[1]{\underline{\vec{#1}}}

\newcommand{\vUT}[1][]{\uvec{U}_T^{#1}}
\newcommand{\vIT}[1][]{\uvec{I}_T^{#1}}

\newcommand{\vUh}[1][]{\uvec{U}_h^{#1}}
\newcommand{\vIh}[1][]{\uvec{I}_h^{#1}}
\newcommand{\vUhD}[1][]{\uvec{U}_{h,0}^{#1}}

\newcommand{\strain}{\tens{\varepsilon}}

\newcommand{\vpT}[1][1]{\vec{p}_T^{#1}}
\newcommand{\vph}[1][1]{\vec{p}_h^{#1}}

\newcommand{\jump}[2][F]{[#2]_{#1}}

\newcommand{\Ndofs}[1][h]{N_{{\rm dofs},#1}}
\newcommand{\Nnz}[1][h]{N_{{\rm nz},#1}}

\title{A low-order nonconforming method for linear elasticity on general meshes}

\author[1]{Michele Botti\footnote{\email{bottieaffini@gmail.com}}}
\author[1]{Daniele A. Di Pietro\footnote{\email{daniele.di-pietro@umontpellier.fr}}}
\author[1,2]{Alessandra Guglielmana\footnote{\email{alessandra.guglielmana@mail.polimi.it}}}
\affil[1]{IMAG, Univ Montpellier, CNRS, Montpellier, France}
\affil[2]{Politecnico di Milano, 20133 Milano, Italy}

\begin{document}
\maketitle

\begin{abstract}
  In this work we construct a low-order nonconforming approximation method for linear elasticity problems supporting general meshes and valid in two and three space dimensions.
  The method is obtained by hacking the Hybrid High-Order method of \cite{Di-Pietro.Ern:15}, that requires the use of polynomials of degree $k\ge 1$ for stability.
  Specifically, we show that coercivity can be recovered for $k=0$ by introducing a novel term that penalises the jumps of the displacement reconstruction across mesh faces.
  This term plays a key role in the fulfillment of a discrete Korn inequality on broken polynomial spaces, for which a novel proof valid for general polyhedral meshes is provided.
  Locking-free error estimates are derived for both the energy- and the $L^2$-norms of the error, that are shown to convergence, for smooth solutions, as $h$ and $h^2$, respectively (here, $h$ denotes the meshsize).
  A thorough numerical validation on a complete panel of two- and three-dimensional test cases is provided.
\end{abstract}


\section{Introduction}

Discretisation methods supporting meshes with general, possibly non standard, element shapes have experienced a vigorous growth over the last few years.
In the context of solid-mechanics, this feature can be useful for several reasons including, e.g., improved robustness to mesh distortion and fracture, local mesh refinement, or the use of hanging nodes for contact and interface problems.
A non-exahustive list of contributions in the context of elasticity problems includes \cite{Hughes.Cottrell.ea:05,Tabarraei.Sukumar:06,Beirao-Da-Veiga:10,Beirao-da-Veiga.Brezzi.ea:13*1,Droniou.Lamichhane:15,Gain.Talischi.ea:14,Di-Pietro.Lemaire:15,Di-Pietro.Ern:15,Botti.Di-Pietro.ea:17,Artioli.Beirao-da-Veiga.ea:17,Koyama.Kikuchi:17,Cockburn.Fu:18,Sevilla.Giacomini.ea:18,Caceres.Gatica.ea:19}; see also references therein.

For large three-dimensional simulations, or whenever one cannot expect the exact solution to be smooth, low-order methods are often privileged in order to reduce the number of unknowns.
It is well-known, however, that low-order Finite Element (FE) approximations are in some cases unsatisfactory:
affine conforming FE methods are not robust in the quasi-incompressible limit owing to their inability to represent non-trivial divergence-free displacement fields;
nonconforming (Crouzeix--Raviart) FE methods, on the other hand, yield unstable discretisations unless appropriate measures are taken; see, e.g., the discussions in \cite{Brenner.Sung:92,Hansbo.Larson:03}.
The underlying reason for this lack of stability is the non-fulfillment of a discrete counterpart of Korn's inequality owing to a poor control of rigid-body motions at mesh faces.
For similar reasons, the stability of Hybrid High-Order (HHO) methods for linear elasticity requires the use of polynomials of degree $k\ge 1$ as unknowns; see \cite[Lemma 4]{Di-Pietro.Ern:15}.
As a matter of fact, as we show in Section \ref{sec:comparison.hho} below, the stability and consistency requirements on the local HHO stabilisation term are incompatible when $k=0$, that is, when piecewise constant polynomials on the mesh and its skeleton are used as discrete unknowns.

In this paper we highlight a modification of the HHO method which recovers stability for $k=0$.
The proposed fix consists in adding a novel term which penalises in a least square sense the jumps of the local affine displacement reconstruction.
This modification is inspired by the Korn inequality on broken polynomial spaces proved in Lemma \ref{lem:korn} below, which appears to be a novel extension of similar results to general polyhedral meshes.
The proof combines the techniques of \cite[Lemma 2.2]{Brenner:03} with the recent results of \cite{Di-Pietro.Ern:12} and \cite{Di-Pietro.Droniou:17} concerning, respectively, the node-averaging operator and local inverse inequalities on polyhedral meshes.
In the context of Crouzeix--Raviart FE approximations of linear elasticity problems on standard meshes, similar jump penalisation terms have been considered in \cite{Hansbo.Larson:03}.

The resulting method has several appealing features:
it is valid in two and three space dimensions, paving the way to unified implementations;
it hinges on a reduced number of unknowns ($15$ for a tetrahedron, $21$ for a hexahedron and, for more general polyhedral shapes, $3$ unknowns per face plus $3$ unknowns inside the element);
it is robust in the quasi-incompressible limit;
it admits a formulation in terms of conservative numerical tractions, which enables its integration in existing Finite Volume simulators (a particularly relevant feature in the context of industrial applications).

We carry out a complete convergence analysis based on the abstract framework of \cite{Di-Pietro.Droniou:18} for methods in fully discrete formulation.
Specifically, we show that the energy and $L^2$-norms of the error converge, respectively, as $h$ and $h^2$ (with $h$ denoting, as usual, the meshsize).
As for the original HHO method of \cite{Di-Pietro.Ern:15}, the error estimates are additionally shown to be robust in the quasi-incompressible limit.
Key to this result is the fact that the gradient of the local displacement reconstruction satisfies a suitable commutation property with the $L^2$-orthogonal projector.
The theoretical results are supported by a thorough numerical investigation, including two- and three-dimensional test cases, as well as a comparison with the original HHO method of \cite{Di-Pietro.Ern:15} on a test case mimicking a mode 1 fracture.

The rest of the paper is organised as follows.
In Section \ref{sec:continuous.setting} we formulate the continuous problem along with the assumptions on the problem data.
In Section \ref{sec:discrete.setting} we establish the discrete setting: after briefly recalling the notion of regular polyhedral mesh, we introduce local and broken polynomial spaces and projectors thereon, and we prove a discrete counterpart of Korn's first inequality on broken polynomial spaces.
In Section \ref{sec:discretisation} we introduce the space of discrete unknowns, define a local affine displacement reconstruction, formulate the discrete bilinear form, discuss the differences with respect to the original HHO bilinear form of \cite{Di-Pietro.Ern:15}, and state the discrete problem.
Section \ref{sec:convergence} addresses the convergence analysis of the method in the energy- and $L^2$-norms, while Section \ref{sec:numerical.tests} contains an exhaustive panel of two- and three-dimensional numerical tests.
Finally, in Section \ref{sec:flux} we show that the method satisfies local balances with equilibrated tractions, for which an explicit expression is provided.


\section{Continuous setting}\label{sec:continuous.setting}

Consider a body which, in its reference configuration, occupies a given region of space $\Omega\subset\Real^d$, $d\in\{2,3\}$.
In what follows, it is assumed that $\Omega$ is a bounded connected open polygonal (if $d=2$) or polyhedral (if $d=3$) set that does not have cracks, i.e., it lies on one side of its boundary $\partial\Omega$.
We are interested in  finding the displacement field $\vec{u}:\Omega\to\Real^d$ of the body when it is subjected to a given force per unit volume $\vec{f}:\Omega\to\Real^d$.
We work in what follows under the small deformation assumption which implies, in particular, that the strain tensor $\strain$ is given by the symmetric part of the gradient of the displacement field, i.e., $\strain=\GRADs\vec{u}$ where, for any vector-valued function $\vec{z}=(z_i)_{1\le i\le d}$ smooth enough, we have set $\GRAD\vec{z}=(\partial_jz_i)_{1\le i,j\le d}$ and $\GRADs\vec{z}\coloneq\frac12\left(\GRAD\vec{z}+\GRAD\vec{z}\trans\right)$.
We further assume, for the sake of simplicity, that the body is clamped along its boundary $\partial\Omega$.
Other standard boundary conditions can be considered up to minor modifications.
The displacement field is obtained by solving the following linear elasticity problem, which expresses the equilibrium between internal stresses and external loads:
Find $\vec{u}:\Omega\to\Real^d$ such that
\begin{subequations}\label{eq:strong}
  \begin{alignat}{2}
    -\DIV(\tens{\sigma}(\GRADs\vec{u})) &= \vec{f}&\qquad&\text{in $\Omega$},\label{eq:strong:pde}
    \\
    \vec{u} &= \vec{0}&\qquad&\text{on $\partial\Omega$},\label{eq:strong:bc}
  \end{alignat}
\end{subequations}
where, denoting by $\Real_{\rm sym}^{d\times d}$ the set of symmetric real-valued $d\times d$ matrices, the mapping $\tens{\sigma}:\Real_{\rm sym}^{d\times d}\to\Real_{\rm sym}^{d\times d}$ represents the strain-stress law.
For isotropic homogeneous materials, the strain-stress law is such that, for any $\tens{\tau}\in\Real_{\rm sym}^{d\times d}$,
\begin{equation}\label{eq:sigma}
  \tens{\sigma}(\tens{\tau}) = 2\mu\tens{\tau} + \lambda\tr(\tens{\tau})\Id,
\end{equation}
where $\tr(\tens{\tau})\coloneq\sum_{i=1}^d\tau_{ii}$ is the trace operator and $\Id$ the $d\times d$ identity matrix.
The real numbers $\mu$ and $\lambda$, which correspond to the Lam\'e coefficients when $d=3$, are assumed such that, for a real number $\alpha>0$,
\begin{equation}\label{eq:lambda.mu.bounds}
  2\mu - d\lambda^-\ge\alpha,
\end{equation}
where $\lambda^-\coloneq\frac12\left(|\lambda|-\lambda\right)$ denotes the negative part of $\lambda$.
In what follows, $\mu$, $\lambda$, the related bound \eqref{eq:lambda.mu.bounds}, and $\vec{f}$ will be collectively referred to as the problem data.

For any open bounded set $X\subset\Omega$, we denote by $({\cdot},{\cdot})_X$ the usual inner product of the space of scalar-valued, square-integrable functions $L^2(X;\Real)$, by $\norm[X]{{\cdot}}$ the corresponding norm, and we adopt the convention that the subscript is omitted whenever $X=\Omega$.
The same notation is used for the spaces of vector- and tensor-valued square-integrable functions $L^2(X;\Real^d)$ and $L^2(X;\Real^{d\times d})$, respectively.
With this notation, a classical weak formulation of problem \eqref{eq:strong} reads:
Find $\vec{u}\in H_0^1(\Omega;\Real^d)$ such that
\begin{equation}\label{eq:weak}
  (\tens{\sigma}(\GRADs\vec{u}),\GRADs\vec{v}) = (\vec{f},\vec{v})
  \qquad\forall\vec{v}\in H_0^1(\Omega;\Real^d),
\end{equation}
where $H_0^1(\Omega;\Real^d)$ classically denotes the space of vector-valued functions that are square-integrable along with all their partial derivatives, and whose traces on $\partial\Omega$ vanish.


\section{Discrete setting}\label{sec:discrete.setting}

\subsection{Mesh}

Throughout the rest of the paper, we will use for the sake of simplicity the three-dimensional nomenclature also when $d=2$, i.e., we will speak of polyhedra and faces rather than polygons and edges.
We consider here meshes corresponding to couples $\Mh\coloneq(\Th,\Fh)$, where $\Th$ is a finite collection of polyhedral elements $T$ such that $h\coloneq\max_{T\in\Th}h_T>0$ with $h_T$ denoting the diameter of $T$, while $\Fh$ is a finite collection of planar faces $F$. It is assumed henceforth that the mesh $\Mh$ matches the geometrical requirements detailed in \cite[Definition 7.2]{Droniou.Eymard.ea:18}; see also~\cite[Section 2]{Di-Pietro.Tittarelli:18}.
This covers, essentially, any reasonable partition of $\Omega$ into polyhedral sets, not necessarily convex or even star-shaped.
For every mesh element $T\in\Th$, we denote by $\Fh[T]$ the subset of $\Fh$ containing the faces that lie on the boundary $\partial T$ of $T$.
Symmetrically, for every face $F\in\Fh$, we denote by $\Th[F]$ the subset of $\Th$ containing the (one or two) mesh elements that share $F$.
For any mesh element $T\in\Th$ and each face $F\in\Fh[T]$, $\normal_{TF}$ is the constant unit normal
vector to $F$ pointing out of $T$.
Boundary faces lying on $\partial\Omega$ and internal faces contained in
$\Omega$ are collected in the sets $\Fhb$ and $\Fhi$, respectively.
For any $F\in\Fhi$, we denote by $T_1$ and $T_2$ the elements of $\Th$ such that $F\subset\partial T_1\cap\partial T_2$.
The numbering of $T_1$ and $T_2$ is assumed arbitrary but fixed, and we set $\normal_F\coloneq\normal_{T_1F}$.
Our focus is on the $h$-convergence analysis, so we consider a sequence of refined meshes that is regular in the sense of~\cite[Definition~3]{Di-Pietro.Tittarelli:18}.  This implies, in particular, that the diameter $h_T$ of a mesh element $T\in\Th$ is comparable to the diameter $h_F$ of each face $F\in\Fh[T]$ uniformly in $h$, and that the number of faces in $\Fh[T]$ is bounded above by an integer $N_\partial$ independent of $h$.

\subsection{Local and broken spaces and projectors}

In order to alleviate the exposition, throughout the rest of the paper we use the abridged notation $a\lesssim b$ for the inequality $a\le Cb$ with real number $C>0$ independent of the meshsize, possibly on the problem data, and, for local inequalities, on the mesh element or face.
We also write $a\simeq b$ for $a\lesssim b$ and $b \lesssim a$.
The dependencies of the hidden constant are further specified whenever needed.

Let $X$ denote a mesh element or face.
For a given integer $l\ge 0$, we denote by $\Poly{l}(X;\Real)$ the space spanned by the restriction to $X$ of $d$-variate, real-valued polynomials of total degree $\le l$.
The corresponding spaces of vector- and tensor-valued functions are respectively denoted by $\Poly{l}(X;\Real^d)$ and $\Poly{l}(X;\Real^{d\times d})$.
A similar notation is used also for the vector and tensor versions of the broken spaces introduced in what follows.
At the global level, we denote by $\Poly{l}(\Th;\Real)$ the space of broken polynomials on $\Th$ whose restriction to every mesh element $T\in\Th$ lies in $\Poly{l}(T;\Real)$, i.e.,
$$
\Poly{l}(\Th;\Real)\coloneq\left\{v_h\in L^2(\Omega;\Real)\st v_{h|T}\in\Poly{l}(T;\Real)\quad\forall T\in\Th\right\}.
$$
We also introduce the broken Sobolev spaces
$$
H^s(\Th;\Real)\coloneq\left\{
v\in L^2(\Omega;\Real)\st v_{|T}\in H^s(T;\Real)\quad\forall T\in\Th
\right\},
$$
which will be used in the error estimates to express the regularity requirements on the exact solution. 
On $H^s(\Th;\Real)$, we define the broken seminorm 
$$
  \seminorm[H^s(\Th;\Real)]{v}\coloneq\left(
  \sum_{T\in\Th}\seminorm[H^s(T;\Real)]{v}^2
  \right)^{\frac12}.
$$

Again denoting by $X$ a mesh element or face, the local $L^2$-orthogonal projector $\lproj[X]{0}:L^2(X;\Real)\to\Poly{0}(X;\Real)$ maps every $v\in L^2(X;\Real)$ onto the constant function equal to its mean value inside $T$, that is,
\begin{equation}\label{eq:lproj}
  \lproj[X]{0}v\coloneq\frac{1}{\meas{X}}\int_X v,
\end{equation}
with $\meas{X}$ denoting the Hausdorff measure of $X$.
The vector and tensor versions of the $L^2$-projector, both denoted by $\vlproj[X]{0}$, are obtained applying $\lproj[X]{0}$ component-wise.
From \cite[Lemmas 3.4 and 3.6]{Di-Pietro.Droniou:17}, it can be deduced that, for any mesh element $T\in\Th$ and any function $v\in H^1(T;\Real)$, the following approximation properties hold:
\begin{equation}\label{eq:lproj:approx}
  \norm[L^2(T;\Real)]{v-\lproj[T]{0}v}
  + h_T^{\frac12}\norm[L^2(\partial T;\Real)]{v-\lproj[T]{0}v}
  \lesssim h_T\seminorm[H^1(T;\Real)]{v},
\end{equation}
where $\partial T$ denotes the boundary of $T$ and the hidden constant is independent of $h$, $T$, and $v$.
The global $L^2$-orthogonal projector $\lproj{0}:L^2(\Omega;\Real)\to\Poly{0}(\Th;\Real)$ is such that, for any $v\in L^2(\Omega;\Real^d)$,
\begin{equation}\label{eq:lproj.h}
  (\lproj{0} v)_{|T}\coloneq\lproj[T]{0} v_{|T}\qquad\forall T\in\Th.
\end{equation}
The vector and tensor versions, both denoted by $\vlproj{0}$, are obtained applying $\lproj{0}$ component-wise.

We will also need the elliptic projector $\eproj[T]{1}:H^1(T;\Real)\to\Poly{1}(T;\Real)$ such that, for all $v\in H^1(T;\Real)$,
\begin{equation}\label{eq:eproj}
  \GRAD\eproj[T]{1}v = \vlproj[T]{0}(\GRAD v)\mbox{ and }
  \frac{1}{\meas{T}}\int_T\eproj[T]{1}v = \frac{1}{\meas{T}}\int_T v.
\end{equation}
The first relation makes sense since $\GRAD\Poly{1}(T;\Real)=\Poly{0}(T;\Real^d)$, and it defines $\eproj[T]{1}v$ up to a constant, which is then fixed by the second relation.
Also in this case, the vector version $\veproj[T]{1}$ of the projector is obtained applying the scalar version component-wise.
The following approximation properties for the elliptic projector are a special case of \cite[Theorems 1.1 and 1.2]{Di-Pietro.Droniou:17*1}:
For all $T\in\Th$ and all $v\in H^2(T;\Real)$,
\begin{equation}\label{eq:eproj:approx}
  \norm[L^2(T;\Real)]{v-\eproj[T]{1}v}
  + h_T^{\frac12}\norm[L^2(\partial T;\Real)]{v-\eproj[T]{1}v}
  \lesssim h_T^2\seminorm[H^2(T;\Real)]{v},
\end{equation}
where the hidden constant is independent of $h$, $T$, and $v$.
For further use, we also define the global elliptic projector $\eproj{1}:H^1(\Th;\Real)\to\Poly{1}(\Th;\Real)$ such that, for any $v\in H^1(\Th;\Real)$,
$$
(\eproj{1} v)_{|T}\coloneq\eproj[T]{1} v_{|T}\qquad\forall T\in\Th.
$$
The vector version $\veproj{1}$ of the global elliptic projector is obtained applying $\eproj{1}$ component-wise.

\subsection{Discrete Korn inequality on broken polynomial spaces}

The stability of our method hinges on a discrete counterpart of Korn's inequality in discrete polynomial spaces stating that the $H^1$-seminorm of a vector-valued broken polynomial function is controlled by a suitably defined strain norm.
The goal of this section is to prove this inequality.

Let us start with some preliminary results.
Recalling that, for any $F\in\Fhi$, we have denoted by $T_1$ and $T_2$ the elements sharing $F$ and assumed that the ordering is arbitrary but fixed, we introduce the jump operator such that, for any function $v$ smooth enough to admit a (possibly two-valued) trace on $F$,
\begin{subequations}\label{eq:jump}
  \begin{equation}
    \jump{v}\coloneq (v_{|T_1})_{|F} - (v_{|T_2})_{|F}.
  \end{equation}
  This operator is extended to boundary faces $F\in\Fhb$ by setting
  \begin{equation}
    \jump{v}\coloneq v_{|F}.
  \end{equation}
\end{subequations}
When applied to vector-valued functions, the jump operator acts componentwise.

Let now $\fTh$ denote a matching simplicial submesh of $\Mh$ in the sense of \cite[Definition 4.2]{Di-Pietro.Tittarelli:18}, and let $\fFh$ be the corresponding set of simplicial faces.
Given an integer $l\ge 1$, we define the node-averaging operator $\Iav[l]:\Poly{l}(\Th;\Real)\to\Poly{l}(\Th;\Real)\cap H_0^1(\Omega)$ such that, for any function $v_h\in\Poly{l}(\Th;\Real)$ and any Lagrange node $V$ of $\fTh$, denoting by $\fTh[V]$ the set of simplices sharing $V$,
$$
(\Iav[l] v_h)(V)\coloneq\begin{cases}
\dfrac{1}{\card(\fTh[V])}\displaystyle\sum_{\tau\in\fTh[V]} (v_h)_{|\tau}(V) & \text{if $V\in\Omega$},
\\
0 & \text{if $V\in\partial\Omega$}.
\end{cases}
$$
The vector-version, denoted by $\vIav[l]$, acts component-wise.
Adapting the reasoning of \cite[Section 5.5.2]{Di-Pietro.Ern:12} (based in turn on \cite{Karakashian.Pascal:03}), we infer that it holds, for all $T\in\Th$, 
\begin{equation}\label{eq:est.Iav.l2}
  \norm[T]{v_h-\Iav[l] v_h}^2\lesssim\sum_{F\in\Fh[\mathcal{V},T]}h_F\norm[F]{\jump{v_h}}^2,
\end{equation}
where $\Fh[\mathcal{V},T]$ denotes the set of faces whose closure has nonempty intersection with the closure of $T$ and the hidden constant is independent of $h$, $T$, and $v_h$.
Combining this result with an inverse inequality (see \cite[Remark A.2]{Di-Pietro.Droniou:17}) we obtain, with hidden constants as before,
$$
\begin{aligned}
  \seminorm[H^1(\Th;\Real)]{v_h-\Iav[l] v_h}^2
  &\lesssim\sum_{T\in\Th} h_T^{-2}\norm[T]{v_h-\Iav[l] v_h}^2
  \\
  &\lesssim\sum_{T\in\Th} h_T^{-2}\sum_{F\in\Fh[\mathcal{V},T]}h_F\norm[F]{\jump{v_h}}^2
  \\
  &\lesssim\sum_{F\in\Fh}\sum_{T\in\Th[\mathcal{V},F]}h_F^{-1}\norm[F]{\jump{v_h}}^2,
\end{aligned}
$$
where we have used \eqref{eq:est.Iav.l2} to pass to the second line while, to pass to the third line, we have invoked the mesh regularity to write $h_F h_T^{-2}\lesssim h_F^{-1}$ and we have exchanged the order of the sums after introducing the notation $\Th[\mathcal{V},F]$ for the set of mesh elements whose closure has nonzero intersection with the closure of $F$.
Using again mesh regularity to infer that $\card(\Th[\mathcal{V},F])$ is bounded uniformly in $h$, we arrive at
\begin{equation}\label{eq:est.Iav.h1}
  \seminorm[H^1(\Th;\Real)]{v_h-\Iav[l] v_h}^2\lesssim\sum_{F\in\Fh}h_F^{-1}\norm[F]{\jump{v_h}}^2.
\end{equation}
We are now ready to prove the discrete Korn inequality.
\begin{lemma}[Discrete Korn inequality]\label{lem:korn}
  Let an integer $l\ge1$ be fixed and set, for all $\vec{v}_h\in\Poly{l}(\Th;\Real^d)$,
   \begin{equation}\label{eq:norm.dG}
    \norm[\strain,h]{\vec{v}_h}\coloneq\left(
    \norm{\GRADsh\vec{v}_h}^2
    + \seminorm[{\rm j},h]{\vec{v}_h}^2
    \right)^{\frac12}
    \mbox{ and }
    \seminorm[\mathrm{j},h]{\vec{v}_h}\coloneq\left(
    \sum_{F\in\Fh}h_F^{-1}\norm[F]{\jump{\vec{v}_h}}^2
    \right)^{\frac12},
  \end{equation}
  where $\GRADsh:H^1(\Th;\Real^d)\to L^2(\Omega;\Real_{\rm sym}^{d\times d})$ is the broken symmetric gradient such that $(\GRADsh\vec{v})_{|T}=\GRADs\vec{v}_{|T}$ for any $T\in\Th$.
  Then, for all $\vec{v}_h\in\Poly{l}(\Th;\Real^d)$, it holds with hidden constant depending only on $\Omega$, $d$, and the mesh regularity parameter:
  \begin{equation}\label{eq:korn}
    \seminorm[H^1(\Th;\Real^d)]{\vec{v}_h}\lesssim\norm[\strain,h]{\vec{v}_h}.
  \end{equation}
\end{lemma}
\begin{proof}
  The proof adapts the arguments of \cite[Lemma 2.2]{Brenner:03}.
  We can write
  $$
  \begin{aligned}
    \seminorm[H^1(\Th;\Real^d)]{\vec{v}_h}^2
    &\lesssim\seminorm[H^1(\Omega;\Real^d)]{\vIav[l]\vec{v}_h}^2 + \seminorm[H^1(\Th;\Real^d)]{\vec{v}_h-\vIav[l]\vec{v}_h}^2
    \\
    &\lesssim\norm{\GRADs\vIav[l]\vec{v}_h}^2 + \seminorm[\mathrm{j},h]{\vec{v}_h}^2
    \\
    &\lesssim\norm{\GRADsh\vec{v}_h}^2 + \norm{\GRADsh(\vIav[l]\vec{v}_h-\vec{v}_h)}^2 + \seminorm[\mathrm{j},h]{\vec{v}_h}^2
    \\
    &\lesssim\norm{\GRADsh\vec{v}_h}^2 + \seminorm[\mathrm{j},h]{\vec{v}_h}^2
    =\norm[\strain,h]{\vec{v}_h}^2,
  \end{aligned}
  $$
  where we have inserted $\pm\vIav[l]\vec{v}_h$ into the seminorm and used a triangle inequality in the first line,
  we have applied the first Korn inequality in $H^1(\Omega;\Real^d)$ to the first term and invoked \eqref{eq:est.Iav.h1} for the second term after recalling the definition \eqref{eq:norm.dG} of the jump seminorm in the second line,
  we have inserted $\pm\GRADsh\vIav[l]\vec{v}_h$ and used a triangle inequality to pass to the third line,
  we have invoked again \eqref{eq:est.Iav.h1} to estimate the second term in the right-hand side to pass to the fourth line, and we have used the definition \eqref{eq:norm.dG} of the strain norm to conclude.
\end{proof}
\begin{remark}[Korn--Poincar\'e inequality]
  Combining the discrete Poincar\'e inequality resulting from \cite[Theorem 6.1]{Di-Pietro.Ern:10} (see also \cite[Theorem 5.3 and Corollary 5.4]{Di-Pietro.Ern:12}) with \eqref{eq:korn}, we infer that it holds, for all $\vec{v}_h\in\Poly{l}(\Th;\Real^d)$,
  \begin{equation}\label{eq:korn-poincare}
    \norm{\vec{v}_h}\lesssim\norm[\strain,h]{\vec{v}_h},
  \end{equation}
  with hidden constant independent of $h$ and $\vec{v}_h$.
\end{remark}

\section{Discretisation}\label{sec:discretisation}

\subsection{Discrete space}

Given a mesh $\Mh=(\Th,\Fh)$, we define the following space of discrete unknowns:
$$
\vUh\coloneq\left\{
\uvec{v}_h=( (\vec{v}_T)_{T\in\Th},(\vec{v}_F)_{F\in\Fh} )\st
\vec{v}_T\in\Poly{0}(T;\Real^d)\quad\forall T\in\Th\mbox{ and }
\vec{v}_F\in\Poly{0}(F;\Real^d)\quad\forall F\in\Fh
\right\}.
$$
For all $\uvec{v}_h\in\vUh$, we denote by $\vec{v}_h\in\Poly{0}(\Th;\Real^d)$ the piecewise constant function obtained by patching element-based unknowns, that is,
\begin{equation}\label{eq:vh}
  (\vec{v}_h)_{|T}\coloneq\vec{v}_T\qquad\forall T\in\Th.
\end{equation}
The restrictions of $\vUh$ and $\uvec{v}_h\in\vUh$ to a generic mesh element $T\in\Th$ are respectively denoted by $\vUT$ and $\uvec{v}_T=(\vec{v}_T,(\vec{v}_F)_{F\in\Fh[T]})$.
The vector of discrete variables corresponding to a smooth function on $\Omega$ is obtained via the global interpolation operator $\vIh:H^1(\Omega;\Real^d)\to\vUh$ such that, for all $\vec{v}\in H^1(\Omega;\Real^d)$,
$$
\vIh\vec{v}\coloneq( (\vlproj[T]{0}\vec{v}_{|T})_{T\in\Th}, (\vlproj[F]{0}\vec{v}_{|F})_{F\in\Fh} ).
$$
Its restriction to a generic mesh element $T\in\Th$ is the local interpolator $\vIT:H^1(T;\Real^d)\to\vUT$ such that, for all $\vec{v}\in H^1(T;\Real^d)$,
\begin{equation}\label{eq:vIT}
  \vIT\vec{v} = (\vlproj[T]{0}\vec{v}, (\vlproj[F]{0}\vec{v}_{|F})_{F\in\Fh[T]}).
\end{equation}
The displacement is sought in the following subspace of $\vUh$ that strongly incorporates the homogeneous Dirichlet boundary condition:
$$ 
\vUhD\coloneq\left\{
\uvec{v}_h\in\vUh\st\vec{v}_F=\vec{0}\quad\forall F\in\Fhb
\right\}.
$$

\subsection{Displacement reconstruction}

Let a mesh element $T\in\Th$ be fixed.
We define the local displacement reconstruction operator $\vpT:\vUT\to\Poly{1}(T;\Real^d)$ such that, for all $\uvec{v}_T\in\vUT$,
\begin{equation}\label{eq:vpT}
  \GRAD\vpT\uvec{v}_T = \sum_{F\in\Fh[T]}\frac{\meas{F}}{\meas{T}}(\vec{v}_F-\vec{v}_T)\otimes\normal_{TF}
  \mbox{ and }
  \frac{1}{\meas{T}}\int_T\vpT\uvec{v}_T=\vec{v}_T.
\end{equation}
\begin{remark}[Explicit expression for the displacement reconstruction operator]
  From \eqref{eq:vpT}, one can infer the following explicit expression for the displacement reconstruction operator: For all $\vec{x}\in T$,
  \begin{equation}\label{eq:vpT:bis}
    \vpT(\vec{x}) = \vec{v}_T
    + \sum_{F\in\Fh[T]}\frac{\meas{F}}{\meas{T}}(\vec{x}-\overline{\vec{x}}_T)\SCAL\normal_{TF}~(\vec{v}_F-\vec{v}_T),
  \end{equation}
  where $\overline{\vec{x}}_T\coloneq\frac{1}{\meas{T}}\int_T\vec{x}$ denotes the centroid of $T$.
\end{remark}
\begin{proposition}[Commutation properties for the displacement reconstruction]
  It holds, for all $\vec{v}\in H^1(T;\Real^d)$,
  \begin{equation}\label{eq:vpT:commutation}
    \GRAD(\vpT\vIT\vec{v})=\tlproj[T]{0}(\GRAD\vec{v})\mbox{ and }
    \vpT(\vIT\vec{v}) = \veproj[T]{1}\vec{v}.  
  \end{equation}
\end{proposition}
\begin{proof}
  Let $\vec{v}\in H^1(T;\Real^d)$.
  Recalling the definition \eqref{eq:vIT} of the local interpolator, we have that
  $$
  \begin{aligned}
    \GRAD\vpT\vIT\vec{v}
    &= \sum_{F\in\Fh[T]}\frac{\meas{F}}{\meas{T}}(\vlproj[F]{0}\vec{v}-\vlproj[T]{0}\vec{v})\otimes\normal_{TF}
    \\
    &= \frac{1}{\meas{T}}\sum_{F\in\Fh[T]}\int_F\vec{v}\otimes\normal_{TF}
    - \frac{1}{\meas{T}}\sum_{F\in\Fh[T]}\int_F\vlproj[T]{0}\vec{v}\otimes\normal_{TF}
    \\
    &= \frac{1}{\meas{T}}\int_T\GRAD\vec{v}
    - \cancel{\frac{1}{\meas{T}}\int_T\GRAD\vlproj[T]{0}\vec{v}},
  \end{aligned}
  $$
  where we have used the definition \eqref{eq:vpT} of the local displacement reconstruction with $\uvec{v}_T=\vIT\vec{v}$ in the first line,
  the definition \eqref{eq:lproj} of the $L^2$-orthogonal projector $\vlproj[F]{0}$ along with the fact that $\vlproj[T]{0}\vec{v}\otimes\normal_{TF}$ is constant over $F$ to pass to the second line,
  the Stokes theorem to pass to the third line and the fact that $\vlproj[T]{0}\vec{v}$ is constant inside $T$ to cancel the second term therein.
  This proves the first relation in \eqref{eq:vpT:commutation}.
  The second relation in \eqref{eq:vpT:commutation} immediately follows accounting for the first and recalling the definition \eqref{eq:eproj} of the elliptic projector after observing that the second relation in \eqref{eq:vpT} gives
  \[
  \pushQED{\qed}
  \frac{1}{\meas{T}}\int_T\vpT\vIT\vec{v}
  = \vlproj[T]{0}\vec{v}
  = \frac{1}{\meas{T}}\int_T\vec{v}.\qedhere
  \]
\end{proof}
To close this section, we define the global displacement reconstruction operator $\vph:\vUh\to\Poly{1}(\Th;\Real^d)$ obtained by patching the local reconstructions:
For all $\uvec{v}_h\in\vUh$,
\begin{equation}\label{eq:vph}
  (\vph\uvec{v}_h)_{|T}\coloneq\vpT\uvec{v}_T\qquad\forall T\in\Th.
\end{equation}

\subsection{Discrete bilinear form}

We define the bilinear form $\mathrm{a}_h:\vUh\times\vUh\to\Real$ such that, for all $\uvec{w}_h,\uvec{v}_h\in\vUh$,
\begin{equation}\label{eq:ah}
  \mathrm{a}_h(\uvec{w}_h,\uvec{v}_h)
  \coloneq
  (\tens{\sigma}(\GRADsh\vph\uvec{w}_h),\GRADsh\vph\uvec{v}_h)
  + (2\mu)~\mathrm{j}_h(\vph[1]\uvec{w}_h,\vph[1]\uvec{v}_h)  
  + (2\mu)~\mathrm{s}_h(\uvec{w}_h,\uvec{v}_h).
\end{equation}
In the above expression, $\mathrm{j}_h: H^1(\Th;\Real^d)\times H^1(\Th;\Real^d)\to\Real$ is the jump penalisation bilinear form such that, for all $\vec{w},\vec{v}\in H^1(\Th;\Real^d)$,
$$
\mathrm{j}_h(\vec{w},\vec{v})
\coloneq\sum_{F\in\Fh}h_F^{-1}(\jump{\vec{w}},\jump{\vec{v}})_F,
$$
while $\mathrm{s}_h:\vUh\times\vUh\to\Real$ is a stabilisation bilinear form defined from local contributions as follows:
\begin{equation}\label{eq:sh}
  \mathrm{s}_h(\uvec{w}_h,\uvec{v}_h)\coloneq\sum_{T\in\Th}\mathrm{s}_T(\uvec{w}_T,\uvec{v}_T)\mbox{ with }
  \mathrm{s}_T(\uvec{w}_T,\uvec{v}_T)
  \coloneq \sum_{F\in\Fh[T]}\frac{\meas{F}}{h_F}\vec{\delta}_{TF}\uvec{w}_T\SCAL\vec{\delta}_{TF}\uvec{v}_T
  \mbox{ for all $T\in\Th$}.
\end{equation}
In the above expression, for all $T\in\Th$ and all $F\in\Fh[T]$, we have introduced the boundary difference operator $\vec{\delta}_{TF}:\vUT\to\Real^d$ is such that, for any $\uvec{v}_T\in\vUT$, 
\begin{equation}\label{eq:dTF}
  \vec{\delta}_{TF}\uvec{v}_T\coloneq\vlproj[F]{0}\vpT\uvec{v}_T-\vec{v}_F.
\end{equation}
It can be proved that the stabilisation bilinear form enjoys the following consistency property:
For all $\vec{w}\in H^1(\Omega;\Real^d)\cap H^2(\Th;\Real^d)$,
\begin{equation}\label{eq:sh:consistency}
  \mathrm{s}_h(\vIh\vec{w},\vIh\vec{w})^{\frac12}\lesssim h\seminorm[H^2(\Th;\Real^d)]{\vec{w}},
\end{equation}
with hidden constant independent of both $h$ and $\vec{w}$.

\subsection{Comparison with the original HHO method and role of the jump penalisation term}\label{sec:comparison.hho}

Compared with the original HHO bilinear form defined by \cite[Eqs. (24)--(26) and (38)]{Di-Pietro.Ern:15} and written for $k=0$, the bilinear form \eqref{eq:ah} includes a novel jump penalisation contribution inspired by the discrete Korn inequality of Lemma \ref{lem:korn}.
This term is needed for stability which, for HHO discretisations of the linear elasticity problem, cannot be achieved through local stabilisation terms for $k=0$.
As a matter of fact, following the ideas of \cite[Section 4.3.1.4]{Di-Pietro.Tittarelli:18}, stability would require the use in \eqref{eq:sh} of a family of local symmetric, positive semidefinite stabilisation bilinear forms $\left\{\mathrm{s}_T\st T\in\Th\right\}$ satisfying the following properties:
\begin{compactenum}[(i)]
\item \emph{Local stability and boundedness.} For all $T\in\Th$ and all $\uvec{v}_T\in\vUT$, with hidden constants independent of $h$, $T$, and $\uvec{v}_T$,
  \begin{equation}\label{eq:sT:norm.equivalence}
    \norm[T]{\GRADs\vpT\uvec{v}_T}^2 + \mathrm{s}_T(\uvec{v}_T,\uvec{v}_T)
    \simeq
    \sum_{F\in\Fh[T]}h_F^{-1}\norm[F]{\vec{v}_F-\vec{v}_T}^2.
  \end{equation}  
\item \emph{Polynomial consistency.} For all $\vec{w}\in\Poly{k+1}(T;\Real^d)$,
  \begin{equation}\label{eq:sT:polynomial.consistency}
    \mathrm{s}_T(\vIT\vec{w},\uvec{v}_T)=0\qquad\forall\uvec{v}_T\in\vUT.
  \end{equation}
\end{compactenum}
Actually, as noticed in \cite[Chapter 7]{Di-Pietro.Droniou:19}, properties \eqref{eq:sT:norm.equivalence} and \eqref{eq:sT:polynomial.consistency} are incompatible.
To see it, assume \eqref{eq:sT:polynomial.consistency}, consider a rigid-body motion $\vec{v}_{{\rm rbm}}$, that is, a function over $\overline{T}$ for which there exist a vector $\vec{t}_{\vec{v}}\in\Real^d$ and a skew-symmetric matrix $\tens{R}_{\vec{v}}\in\Real^{d\times d}$ such that, for any $\vec{x}\in\overline{T}$, $\vec{v}_{{\rm rbm}}(\vec{x}) = \vec{t}_{\vec{v}} + \tens{R}_{\vec{v}}\vec{x}$.
Take now $\uvec{v}_T=\vIT\vec{v}_{{\rm rbm}}$.
Since $\vec{v}_{{\rm rbm}}\in\Poly{1}(T;\Real^d)$, the first relation in \eqref{eq:vpT:commutation} shows that $\GRAD\vpT\uvec{v}_T=\vlproj[T]{0}(\GRAD\vec{v}_{{\rm rbm}})=\GRAD\vec{v}_{{\rm rbm}}=\tens{R}_{\vec{v}}$ so that, in particular, $\GRAD\vpT\uvec{v}_T$ is skew-symmetric. Hence, $\GRADs\vpT\uvec{v}_T=\tens{0}$.
Moreover, by \eqref{eq:sT:polynomial.consistency}, $\mathrm{s}_T(\uvec{v}_T,\uvec{v}_T)=\mathrm{s}_T(\vIT[0]\vec{v}_{{\rm rbm}},\uvec{v}_T)=0$, again because $\vec{v}_{\rm rbm}\in\Poly{1}(T;\Real^d)$.
Hence, the left-hand side of \eqref{eq:sT:norm.equivalence} vanishes for all $\uvec{v}_T=\vIT\vec{v}_{{\rm rbm}}$ with $\vec{v}_{{\rm rbm}}$ rigid-body motion.
It is, however, easy to construct a rigid-body motion $\vec{v}_{{\rm rbm}}$ such that the right-hand side does not vanish, which shows that \eqref{eq:sT:norm.equivalence} cannot hold.
For this reason, the assumption that the discrete unknowns are at least piecewise affine is required in the original HHO method; see \cite[Section 4]{Di-Pietro.Ern:15}.
Notice that the choice of $\mathrm{s}_T$ in \eqref{eq:sh} retains the polynomial consistency property \eqref{eq:sT:polynomial.consistency}, which is crucial to prove \eqref{eq:sh:consistency}.

We next discuss how the stability property modifies for $k=0$.
To this end, recalling the definitions \eqref{eq:norm.dG} of the double-bar strain norm $\norm[\strain,h]{{\cdot}}$ and \eqref{eq:sh} of the stabilisation bilinear form we introduce the triple-bar strain norm such that, for any $\uvec{v}_h\in\vUh$,
\begin{equation}\label{eq:tnorm.strain.h}
  \tnorm[\strain,h]{\uvec{v}_h}
  \coloneq\left(
  \norm[\strain,h]{\vph\uvec{v}_h}^2 + \seminorm[\mathrm{s},h]{\uvec{v}_h}^2
  \right)^{\frac12}\mbox{ with }
  \seminorm[\mathrm{s},h]{\uvec{v}_h}\coloneq\mathrm{s}_h(\uvec{v}_h,\uvec{v}_h)^{\frac12}.
\end{equation}
\begin{lemma}[Global stability and boundedness]
  For all $\uvec{v}_h\in\vUhD$ it holds
  \begin{equation}\label{eq:global.norm.equivalence}
    \norm{\GRADsh\vph\uvec{v}_h}^2 + \seminorm[\mathrm{s},h]{\uvec{v}_h}^2
    \lesssim
    \sum_{T\in\Th}\sum_{F\in\Fh[T]}h_F^{-1}\norm[F]{\vec{v}_F-\vec{v}_T}^2
    \lesssim
    \tnorm[\strain,h]{\uvec{v}_h}^2,
  \end{equation}
  with hidden constant independent of both $h$ and $\uvec{v}_h$.
\end{lemma}
\begin{proof}
  It follows from \cite[Lemma 4]{Di-Pietro.Ern.ea:14} that
  \begin{equation}\label{eq:global.norm.equivalence:1}
    \norm{\GRADh\vph\uvec{v}_h}^2 + \seminorm[\mathrm{s},h]{\uvec{v}_h}^2
    \simeq
    \sum_{T\in\Th}\sum_{F\in\Fh[T]}h_F^{-1}\norm[F]{\vec{v}_F-\vec{v}_T}^2,
  \end{equation}
  where $\GRADh: H^1(\Th;\Real^d)\to L^2(\Omega;\Real^{d\times d})$ is the broken gradient such that $(\GRADh\vec{v})_{|T}=\GRAD\vec{v}_{|T}$ for any $T\in\Th$.
  On the other hand, using the definition of the symmetric gradient for the first bound and Korn's inequality \eqref{eq:korn} for the second, we can write
  \begin{equation}\label{eq:global.norm.equivalence:2}
    \norm{\GRADsh\vph\uvec{v}_h}^2
    \lesssim\norm{\GRADh\vph\uvec{v}_h}^2
    \lesssim\norm{\GRADsh\vph\uvec{v}_h}^2 + \seminorm[\mathrm{j},h]{\vph\vec{v}_h}^2.
  \end{equation}
  Combining \eqref{eq:global.norm.equivalence:2} with \eqref{eq:global.norm.equivalence:1} yields the result.
\end{proof}

\subsection{Discrete problem}

The low-order scheme for the approximation of problem \eqref{eq:strong} reads:
Find $\uvec{u}_h\in\vUhD$ such that
\begin{equation}\label{eq:discrete}
  \mathrm{a}_h(\uvec{u}_h,\uvec{v}_h)
  = (\vec{f},\vec{v}_h)\qquad\forall\uvec{v}_h\in\vUhD.
\end{equation}
Using the coercivity of the bilinear form $\mathrm{a}_h$ proved in Lemma \ref{lem:ah} below together with the discrete Korn inequality \eqref{eq:korn-poincare}, we infer that the discrete problem is well-posed and the a priori estimate $\tnorm[\strain,h]{\uvec{u}_h}\lesssim \alpha^{-\frac12}\norm{\vec{f}}$ holds for the discrete solution, with hidden constant independent of both $h$ and of the problem data, and triple-bar strain seminorm defined by \eqref{eq:tnorm.strain.h} below.
\begin{remark}[Static condensation for problem \eqref{eq:discrete}]\label{rem:static.condensation}
  The jump stabilisation introduces a direct link among discrete unknowns attached to neighbouring mesh elements.
  As a result, static condensation of element-based unknowns no longer appears to be an interesting option.
\end{remark}


\section{Convergence analysis}\label{sec:convergence}

In this section, after studying the properties of the discrete bilinear form $\mathrm{a}_h$, we prove a priori estimates for the error in the energy- and $L^2$-norms.

\subsection{Properties of the discrete bilinear form}

\begin{lemma}[Properties of $\mathrm{a}_h$]\label{lem:ah}
  The bilinear form $\mathrm{a}_h$ enjoys the following properties:
  \begin{compactenum}[(i)]
  \item \emph{Stability and boundedness.} Recalling the definition \eqref{eq:tnorm.strain.h} of the triple-bar strain norm and the bound \eqref{eq:lambda.mu.bounds} on Lam\'e's coefficients, for all $\uvec{v}_h\in\vUh$ it holds 
    \begin{equation}\label{eq:ah:stability}
      \alpha\tnorm[\strain,h]{\uvec{v}_h}^2
      \lesssim\norm[\mathrm{a},h]{\uvec{v}_h}^2
      \lesssim\left(2\mu+d|\lambda|\right)\tnorm[\strain,h]{\uvec{v}_h}^2
      \mbox{ with }
      \norm[\mathrm{a},h]{\uvec{v}_h}\coloneq\mathrm{a}_h(\uvec{v}_h,\uvec{v}_h)^{\frac12},
    \end{equation}
    where the hidden constants are independent of both $h$ and the problem data.
  \item \emph{Consistency.} It holds for all $\vec{w}\in H_0^1(\Omega;\Real^d)\cap H^2(\Th;\Real^d)$ such that $\DIV\tens{\sigma}(\GRADs\vec{w})\in L^2(\Omega;\Real^d)$,
    \begin{equation}\label{eq:ah:consistency}
      \tnorm[\strain,h,*]{\Cerr{\vec{w}}{\cdot}}
      \lesssim h\left(
      2\mu\seminorm[H^2(\Th;\Real^d)]{\vec{w}}
      + \seminorm[H^1(\Th;\Real)]{\lambda\DIV\vec{w}}
      \right),
    \end{equation}
    where the hidden constant is independent of $\vec{w}$, $h$, and of the problem data, the linear form 
    $\Cerr{\vec{w}}{{\cdot}}:\vUhD\to\Real$ representing the consistency error is such that, for all 
    $\uvec{v}_h\in\vUhD$,
    \begin{equation}\label{eq:Eh}
      \Cerr{\vec{w}}{\uvec{v}_h}
      \coloneq
      -(\DIV\tens{\sigma}(\GRADs\vec{w}),\vec{v}_h)
      - \mathrm{a}_h(\vIh\vec{w},\uvec{v}_h),
    \end{equation}
    and its dual norm is given by
    $$
    \tnorm[\strain,h,*]{\Cerr{\vec{w}}{\cdot}}\coloneq
    \sup_{\uvec{v}_h\in\vUhD,\tnorm[\strain,h]{\uvec{v}_h}=1}\left|\Cerr{\vec{w}}{\uvec{v}_h}\right|.
    $$
  \end{compactenum}
\end{lemma}
\begin{proof}
  (i) \emph{Stability and boundedness.}
  Let $\uvec{v}_h\in\vUh$. 
  We recall the Frobenius product such that, for all $\tens{\tau},\tens{\eta}\in\Real^{d\times d}$, $\tens{\tau}\SSCAL\tens{\eta}\coloneq\sum_{i=1}^d\sum_{j=1}^d\tau_{ij}\eta_{ij}$ with corresponding norm $\norm[\rm F]{\tens{\tau}}\coloneq(\tens{\tau}\SSCAL\tens{\tau})^{\frac12}$.
  Writing \eqref{eq:ah} for $\uvec{w}_h=\uvec{v}_h$, using the assumption \eqref{eq:lambda.mu.bounds} on Lam\'e's parameters to infer that $\tens{\sigma}(\tens{\tau})\SSCAL\tens{\tau}\ge\alpha\norm[\rm F]{\tens{\tau}}^2$ for any $\tens{\tau}\in\Real_{\rm sym}^{d\times d}$, recalling the definitions \eqref{eq:norm.dG} and \eqref{eq:tnorm.strain.h} of the double- and triple-bar strain norms, and observing that $2\mu\ge\alpha$, the first inequality in \eqref{eq:ah:stability} follows.
  The second inequality can be obtained in a similar way: write \eqref{eq:ah} for $\uvec{w}_h=\uvec{v}_h$, observe that $|\tens{\sigma}(\tens{\tau})\SSCAL\tens{\tau}|\le(2\mu + d|\lambda|)\norm[\rm F]{\tens{\tau}}^2$ for any $\tens{\tau}\in\Real_{\rm sym}^{d\times d}$, and use again \eqref{eq:norm.dG} and \eqref{eq:tnorm.strain.h}.
  \medskip\\
  (ii) \emph{Consistency.}
  Let $\uvec{v}_h\in\vUhD$.
  We reformulate the components of the consistency error.
  Integrating by parts element by element, we infer that
  $$
  -(\DIV\tens{\sigma}(\GRADs\vec{w}),\vec{v}_h)
  = \sum_{T\in\Th}\sum_{F\in\Fh[T]}(\tens{\sigma}(\GRADs\vec{w})_{|T}\normal_{TF},\vec{v}_F-\vec{v}_T)_F,
  $$
  where we have used the continuity of normal tractions across interfaces together with the fact that boundary unknowns are set to zero in $\vUhD$  to insert $\vec{v}_F$ into the right-hand side.
  To reformulate the second term in \eqref{eq:Eh}, in the expression \eqref{eq:ah} of $\mathrm{a}_h$ we use the first property in \eqref{eq:vpT:commutation} together with the linearity of the strain-stress law $\tens{\sigma}$ to write, for all $T\in\Th$, $\tens{\sigma}(\GRADs\vpT\vIT\vec{w})=\tens{\sigma}(\tlproj[T]{0}(\GRADs\vec{w}))=\tlproj[T]{0}(\tens{\sigma}(\GRADs\vec{w}))$ and obtain
  $$
  \mathrm{a}_h(\vIh\vec{w},\uvec{v}_h)
  = \sum_{T\in\Th}(\tlproj[T]{0}(\tens{\sigma}(\GRADs\vec{w})),\GRADs\vpT\uvec{v}_T)_T
  + (2\mu)~\mathrm{j}_h(\vph\vIh\vec{w},\vph\uvec{v}_h)  
  + (2\mu)~\mathrm{s}_h(\vIh\vec{w},\uvec{v}_h).
  $$
  After expanding, for all $T\in\Th$, $\GRADs\vpT\uvec{v}_T$ according to its definition \eqref{eq:vpT}, we deduce that
  $$
  \mathrm{a}_h(\vIh\vec{w},\uvec{v}_h)
  = \sum_{T\in\Th}\sum_{F\in\Fh[T]}(\tlproj[T]{0}(\tens{\sigma}(\GRADs\vec{w}))\normal_{TF},\vec{v}_F-\vec{v}_T)_F
  + (2\mu)~\mathrm{j}_h(\vph\vIh\vec{w},\vph\uvec{v}_h)  
  + (2\mu)~\mathrm{s}_h(\vIh\vec{w},\uvec{v}_h).
  $$
  Plugging the above relations into the expression \eqref{eq:Eh} of the consistency error, passing to absolute values, using a generalised H\"older inequality with exponents $(2,\infty,2)$ along with $\norm[L^\infty(F;\Real^d)]{\normal_{TF}}\le 1$ and $h_F\le h_T$ for the first term in the right-hand side, and Cauchy--Schwarz inequalities for the remaining terms, we get
  \begin{equation}\label{eq:consistency:basic}
    \begin{aligned}
      \left|\Cerr{\vec{w}}{\uvec{v}_h}\right|    
      &=
      \underbrace{%
        \left(
        \sum_{T\in\Th}h_T\norm[\partial T]{\tens{\sigma}(\GRADs\vec{w})_{|T}-\tlproj[T]{0}(\tens{\sigma}(\GRADs\vec{w}))}^2
        \right)^{\frac12}%
      }_{\term_1}
      \left(
      \sum_{T\in\Th}\sum_{F\in\Fh[T]}h_F^{-1}\norm[F]{\vec{v}_F-\vec{v}_T}^2
      \right)^{\frac12}
      \\
      &\qquad
      + \underbrace{%
        (2\mu)\seminorm[\mathrm{j},h]{\vph\vIh\vec{w}}%
      }_{\term_2}
      ~\seminorm[\mathrm{j},h]{\vph\uvec{v}_h}
      + \underbrace{%
        (2\mu)\seminorm[\mathrm{s},h]{\vIh\vec{w}}%
      }_{\term_3}
      ~\seminorm[\mathrm{s},h]{\uvec{v}_h}      
      \\
      &\lesssim 
      \left(\term_1+\term_2+\term_3\right)\tnorm[\strain,h]{\uvec{v}_h},
    \end{aligned}
  \end{equation}
  where we have used the second inequality in \eqref{eq:global.norm.equivalence} together with the definition \eqref{eq:tnorm.strain.h} of the triple-bar strain norm to conclude.
  Recalling the expression \eqref{eq:sigma} of the strain-stress law, we get, for any $T\in\Th$,
  \begin{equation}\label{eq:est.sigma}
    \begin{aligned}
      h_T^{\frac12}\norm[\partial T]{\tens{\sigma}(\GRADs\vec{w})_{|T}{-}\tlproj[T]{0}(\tens{\sigma}(\GRADs\vec{w}))}
      &\le (2\mu)h_T^{\frac12}\norm[\partial T]{\GRADs\vec{w}{-}(\tlproj[T]{0}\GRADs\vec{w})}
      + h_T^{\frac12}\norm[\partial T]{\lambda\DIV\vec{w}{-}\lproj[T]{0}(\lambda\DIV\vec{w})}
      \\
      &\lesssim
      h\left( (2\mu)\seminorm[H^2(T;\Real^d)]{\vec{w}} + \seminorm[H^1(T;\Real)]{\lambda\DIV\vec{w}}\right),
    \end{aligned}
  \end{equation}
  where we have used the approximation properties \eqref{eq:lproj:approx} of the $L^2$-orthogonal projector along with $h_T\le h$ to conclude.
  Using the above estimate, we infer for the first term
  \begin{equation}\label{eq:consistency:T1}
    \term_1\lesssim h\left(
    (2\mu)\seminorm[H^2(\Th;\Real^d)]{\vec{w}}
    + \seminorm[H^1(\Th;\Real)]{\lambda\DIV\vec{w}}
    \right).
  \end{equation}  
  Moving to the second term, we start by observing that
  $$
  \begin{aligned}
    \term_2^2
    &= (2\mu)^2~\seminorm[\mathrm{j},h]{\veproj{1}\vec{w}}^2
    \\
    &= (2\mu)^2\sum_{F\in\Fh}h_F^{-1}\norm[F]{\jump{\veproj{1}\vec{w}}}^2
    \\
    &= (2\mu)^2\sum_{F\in\Fh}h_F^{-1}\norm[F]{\jump{\veproj{1}\vec{w}-\vec{w}}}^2
    \\
    &\lesssim (2\mu)^2\sum_{F\in\Fh}\sum_{T\in\Th[F]}h_F^{-1}\norm[F]{\veproj[T]{1}\vec{w}-\vec{w}_{|T}}^2
    \\
    &\lesssim (2\mu)^2\sum_{T\in\Th}h_T^{-1}\norm[\partial T]{\veproj[T]{1}\vec{w}-\vec{w}_{|T}}^2
  \end{aligned}
  $$
  where we have used, in this order, the second relation in \eqref{eq:vpT:commutation}, the definition \eqref{eq:norm.dG} of the jump seminorm, the fact that the jumps of $\vec{w}$ vanish across any $F\in\Fh$, the definition \eqref{eq:jump} of the jump operator together with the triangle inequality, and the relation 
\begin{equation}  
  \label{eq:sumTF=sumFT}
  \sum_{T\in\Th}\sum_{F\in\Fh[T]} \bullet =\sum_{F\in\Fh}\sum_{T\in\Th[F]} \bullet
\end{equation}
to exchange the sums over elements and faces.
Hence, using the approximation properties \eqref{eq:eproj:approx} of the elliptic projector, $h_T\le h$, and taking the square root, we arrive at
  \begin{equation}\label{eq:consistency:T2}
    \term_2 \lesssim (2\mu) h \seminorm[H^2(\Th;\Real^d)]{\vec{w}}.
  \end{equation}
  For the third term, \eqref{eq:sh:consistency} readily gives
  \begin{equation}\label{eq:consistency:T3}
    \term_3\lesssim (2\mu) h\seminorm[H^2(\Th;\Real^d)]{\vec{w}}.
  \end{equation}  
  Plugging \eqref{eq:consistency:T1}, \eqref{eq:consistency:T2}, and \eqref{eq:consistency:T3} into \eqref{eq:consistency:basic} and passing to the supremum yields \eqref{eq:ah:consistency}.
\end{proof}

\subsection{Energy error estimate}

\begin{theorem}[Energy error estimate]\label{thm:err.est}
  Let $\vec{u}\in H_0^1(\Omega;\Real^d)$ denote the unique solution to \eqref{eq:weak}, for which we assume the additional regularity $\vec{u}\in H^2(\Th;\Real^d)$.
  For all $h\in{\cal H}$, let $\uvec{u}_h\in\vUhD$ denote the unique solution to \eqref{eq:discrete}.
  Then,
  \begin{equation}\label{eq:err.est}
    \tnorm[\strain,h]{\uvec{u}_h-\vIh\vec{u}}
    \lesssim \alpha^{-1} h\left(
    (2\mu)\seminorm[H^2(\Th;\Real^d)]{\vec{u}}
    + \seminorm[H^1(\Th;\Real)]{\lambda\DIV\vec{u}}
    \right),
  \end{equation}
  with hidden constant independent of $h$, $\vec{u}$, and of the problem data.
\end{theorem}
\begin{proof}
  Applying to the present setting the results of \cite[Theorem 10]{Di-Pietro.Droniou:18} gives the abstract estimate
  $$
  \tnorm[\strain,h]{\uvec{u}_h-\vIh\vec{u}}
  \le\alpha^{-1}\tnorm[\strain,h,*]{\Cerr{\vec{u}}{\cdot}}.
  $$
  Using the assumed regularity for the exact solution to invoke \eqref{eq:ah:consistency}, \eqref{eq:err.est} follows.
\end{proof}

\begin{remark}[Robustness in the quasi-incompressible limit]\label{rem:robustness}
  In the numerical approximation of linear elasticity problems, a key point consists in devising schemes that are robust in the quasi incompressible limit corresponding to $\frac{\lambda}{2\mu}\gg 1$ (which requires, in particular $\lambda\ge 0$).
  From a mathematical perspective, this property is expressed by the fact that the error estimates are uniform in $\lambda$.
  For $d=2$ and $\Omega$ convex, it is proved, e.g., in \cite[Lemma 2.2]{Brenner.Sung:92} that
  \begin{equation}\label{eq:a-priori}
    (2\mu)\norm[H^2(\Omega;\Real^d)]{\vec{u}}
    + \norm[H^1(\Omega;\Real)]{\lambda\DIV\vec{u}}
    \lesssim\norm{\vec{f}},
  \end{equation}
  with hidden constant possibly depending on $\Omega$ and $\mu$ but independent of $\lambda$.
  This result can be extended to $d=3$ reasoning as in the above reference and accounting for the regularity estimates for the Stokes problem derived in \cite[Theorem 3]{Amrouche.Girault:91}.  
  Plugging \eqref{eq:a-priori} into \eqref{eq:err.est} and observing that, when $\lambda\ge 0$, we can take $\alpha=2\mu$ (cf. \eqref{eq:lambda.mu.bounds}), we can write, with hidden constant independent of both $h$ and $\lambda$,
  \begin{equation}\label{eq:en.err.est.robust}
    \tnorm[\strain,h]{\uvec{u}_h-\vIh\vec{u}}\lesssim h \norm{\vec{f}},
  \end{equation}
  which shows that our error estimate \eqref{eq:err.est} is uniform in $\lambda$.
  The key point to obtain robustness is the first commutation property in \eqref{eq:vpT:commutation}, which is used to estimate the term $\mathcal{T}_1$ in the proof of Lemma \ref{lem:ah}.
\end{remark}

\begin{remark}[Quasi-optimality of the error estimate]
  It follows from the second inequality in \eqref{eq:ah:stability} that the bilinear form $\mathrm{a}_h$ is bounded with boundedness constant independent of $h$.
  Hence, following \cite[Remark 11]{Di-Pietro.Droniou:18}, the error estimate \eqref{eq:err.est} is quasi-optimal.
\end{remark}

\begin{remark}[{Energy estimate in the $\norm[\mathrm{a},h]{{\cdot}}$-norm for $\lambda\ge 0$}]\label{rem:err.est:norm.a.h}
  When $\lambda\ge0$, a consistency estimate in $h$ holds for $\norm[\mathrm{a},h,*]{\Cerr{\vec{w}}{\cdot}}$, the norm of the consistency error linear form dual to $\norm[\mathrm{a},h]{{\cdot}}$ (see \eqref{eq:ah:stability}).
  To see it, observe that, from \eqref{eq:consistency:basic} together with $(2\mu)^{\frac12}\tnorm[\strain,h]{\uvec{v}_h}\le\norm[\mathrm{a},h]{\uvec{v}_h}$ (a consequence of the assumption $\lambda\ge 0$), it follows $\left|\Cerr{\vec{w}}{\uvec{v}_h}\right|\lesssim\left(\term_1+\term_2+\term_3\right)(2\mu)^{-\frac12}\norm[\mathrm{a},h]{\uvec{v}_h}$.
  Hence, passing to the supremum over $\big\{\uvec{v}_h\in\vUhD\st\norm[\mathrm{a},h]{\uvec{v}_h}=1\big\}$, we infer
  $$
  \norm[\mathrm{a},h,*]{\Cerr{\vec{w}}{\cdot}}
  \lesssim h\left((2\mu)^{\frac12}\seminorm[H^2(\Th;\Real^d)]{\vec{w}}
                 +(2\mu)^{-\frac12}\seminorm[H^1(\Th;\Real)]{\lambda\DIV\vec{w}}\right).
  $$
  Invoking again \cite[Theorem 10]{Di-Pietro.Droniou:18}, this time with $\vUhD$ equipped with the $\norm[\mathrm{a},h]{{\cdot}}$-norm, it is inferred
$$
    \norm[\mathrm{a},h]{\uvec{u}_h-\vIh\vec{u}}
    \lesssim h\left(
    (2\mu)^{\frac12}\seminorm[H^2(\Th;\Real^d)]{\vec{u}}
    + (2\mu)^{-\frac12}\seminorm[H^1(\Th;\Real)]{\lambda\DIV\vec{u}}
    \right),
$$
  with hidden constant having the same dependencies as in \eqref{eq:err.est}.
\end{remark}

\subsection{Improved $L^2$-error estimate}

Combining the discrete Korn--Poincar\'e inequality \eqref{eq:korn-poincare} with the error estimate \eqref{eq:err.est}, we can infer an estimate in $h$ for the $L^2$-norm of the displacement error $\norm{\vec{u}_h-\vlproj{0}\vec{u}}$, where we remind the reader that $\vec{u}_h$ is defined according to \eqref{eq:vh} as the broken polynomial obtained patching element unknowns, while $\vlproj{0}\vec{u}$ is the $L^2$-orthogonal projection of the exact solution on $\Poly{0}(\Th;\Real^d)$.
It is well-known, however, that improved $L^2$-error estimates can be derived in the context of HHO methods when elliptic regularity holds.
In this section, we show that the same is true for the low-order method considered in this work.
For the sake of simplicity, we assume throughout this section that
$$
\lambda\ge 0.
$$
This assumption could be removed, but we keep it here to simplify the discussion and point out the robustness in the quasi-incompressible limit.
Recalling the discussion in Remark \ref{rem:robustness}, elliptic regularity for our problem entails that, for all $\vec{g}\in L^2(\Omega;\Real^d)$, the unique solution of the (dual) problem:
Find $\vec{z}_{\vec{g}}\in H_0^1(\Omega;\Real^d)$ such that
\begin{equation}\label{eq:dual}
  (\tens{\sigma}(\GRADs\vec{z}_{\vec{g}}),\GRADs\vec{v}) = (\vec{g},\vec{v})
  \qquad\forall\vec{v}\in H_0^1(\Omega;\Real^d)
\end{equation}
satisfies the a priori estimate
\begin{equation}\label{eq:dual:a-priori}
  (2\mu)~\norm[H^2(\Omega;\Real^d)]{\vec{z}_{\vec{g}}}
  + \norm[H^1(\Omega;\Real)]{\lambda\DIV\vec{z}_{\vec{g}}}
  \lesssim\norm{\vec{g}},
\end{equation}
with hidden constant only depending on $\Omega$ and $\mu$.

\begin{theorem}[Improved $L^2$-error estimate]\label{thm:l2.err.est}
  Under the assumptions and notations of Theorem \ref{thm:err.est}, and further assuming $\lambda\ge 0$, elliptic regularity, and $\vec{f}\in H^1(\Th;\Real^d)$, it holds that
  \begin{equation}\label{eq:l2.err.est}
    \norm{\vec{u}_h-\vlproj{0}\vec{u}}
    \lesssim h^2\norm[H^1(\Th;\Real^d)]{\vec{f}},
  \end{equation}
  where the hidden constant is independent of both $h$ and $\lambda$ (but possibly depends on $\mu$).
\end{theorem}
\begin{proof}
  Inside the proof, hidden constants have the same dependencies as in \eqref{eq:l2.err.est}.
  Applying the results of \cite[Theorem 13]{Di-Pietro.Droniou:18} to the present setting gives the basic estimate
  \begin{equation}\label{eq:l2.basic.est}
    \norm{\vec{u}_h-\vlproj{0}\vec{u}}
    \le
    \tnorm[\strain,h]{\uvec{u}_h-\vIh\vec{u}}
         {\sup_{\vec{g}\in L^2(\Omega;\Real^d),\norm{\vec{g}}\le 1}\tnorm[\strain,h,*]{\Cerr{\vec{z}_{\vec{g}}}{\cdot}}}
         + {\sup_{\vec{g}\in L^2(\Omega;\Real^d),\norm{\vec{g}}\le 1}\Cerr{\vec{u}}{\vIh\vec{z}_{\vec{g}}}}.
  \end{equation}
  We proceed to bound the addends in the right-hand side, denoted for the sake of brevity $\term_1$ and $\term_2$.
  \medskip\\
  (i) \emph{Estimate of $\term_1$.}
  Since $\vec{z}_{\vec{g}}\in H_0^1(\Omega;\Real^d)\cap H^2(\Omega;\Real^d)$, the consistency estimate \eqref{eq:ah:consistency} followed by the elliptic regularity bound \eqref{eq:dual:a-priori} yield, for any $\vec{g}\in L^2(\Omega;\Real^d)$,
  $$
  \tnorm[\strain,h,*]{\Cerr{\vec{z}_{\vec{g}}}{\cdot}}
  \lesssim h\left(
  (2\mu)\seminorm[H^2(\Th;\Real^d)]{\vec{z}_{\vec{g}}}
  + \seminorm[H^1(\Th;\Real)]{\lambda\DIV\vec{z}_{\vec{g}}}  
  \right)
  \lesssim h\norm{\vec{g}}.
  $$
  Combined with the energy error estimate \eqref{eq:en.err.est.robust}, this yields
  \begin{equation}\label{eq:l2.err.est:T1}
    \term_1\lesssim h^2\norm{\vec{f}}.
  \end{equation}
  \\
  (ii) \emph{Estimate of $\term_2$.}
  Recalling the expression \eqref{eq:Eh} of the consistency error, expanding the bilinear form $\mathrm{a}_h$ according to its definition \eqref{eq:ah} with $\uvec{w}_h=\vIh\vec{u}$ and $\uvec{v}_h=\vIh\vec{z}_{\vec{g}}$, and invoking  \eqref{eq:vpT:commutation} to replace $\vph\vIh$ with $\veproj{1}$ and $\GRADsh\vph\vIh$ with $\tlproj{0}\GRADs$, we can write
  \begin{multline}\label{eq:l2.err.est:T2:1}
    \Cerr{\vec{u}}{\vIh\vec{z}_{\vec{g}}}
    = (-\DIV\tens{\sigma}(\GRADs\vec{u}),\vlproj{0}\vec{z}_{\vec{g}})
    - (\tlproj{0}(\tens{\sigma}(\GRADs\vec{u})),\tlproj{0}\GRADs\vec{z}_{\vec{g}})
    \\
    -(2\mu)~\mathrm{j}_h(\veproj{1}\vec{u},\veproj{1}\vec{z}_{\vec{g}})
    -(2\mu)~\mathrm{s}_h(\vIh\vec{u},\vIh\vec{z}_{\vec{g}}).
  \end{multline}
  We have that
  $$
  \begin{aligned}
    -(\DIV\tens{\sigma}(\GRADs\vec{u}),\vlproj{0}\vec{z}_{\vec{g}})
    =
    (\vec{f},\vlproj{0}\vec{z}_{\vec{g}})
    &=
    (\vlproj{0}\vec{f},\vec{z}_{\vec{g}})
    \\
    &=
    (\vlproj{0}\vec{f}-\vec{f},\vec{z}_{\vec{g}})
    + (\tens{\sigma}(\GRADs\vec{u}),\GRADs\vec{z}_{\vec{g}})
    \\
    &=
    (\vlproj{0}\vec{f}-\vec{f},\vec{z}_{\vec{g}}-\vlproj{0}\vec{z}_{\vec{g}})
    + (\tens{\sigma}(\GRADs\vec{u}),\GRADs\vec{z}_{\vec{g}}),
  \end{aligned}
  $$
  where we have used the fact that \eqref{eq:strong:pde} holds almost everywhere in $\Omega$ to replace $-\DIV\tens{\sigma}(\GRADs\vec{u})$ with $\vec{f}$ along with the definitions \eqref{eq:lproj.h} and \eqref{eq:lproj} of the global and local $L^2$-orthogonal projectors in the first line,
  we have added the quantity $(\vec{f},\vec{z}_{\vec{g}})-(\tens{\sigma}(\GRADs\vec{u}),\GRADs\vec{z}_{\vec{g}})=0$ (see \eqref{eq:weak}) in the second line,
  while, to pass to the third line, we have used the fact that, by definition of $\vlproj{0}$, the function $(\vlproj{0}\vec{f}-\vec{f})$ is $L^2(\Omega;\Real^d)$-orthogonal to $\Poly{0}(\Th;\Real^d)$ to insert $\vlproj{0}\vec{z}_{\vec{g}}$ into the first term.
  The Cauchy--Schwarz inequality and the approximation property \eqref{eq:lproj:approx} of the $L^2$-orthogonal projector inside each mesh element yield for the first term in the right-hand side
  \begin{equation}\label{eq:T2:1}
    \left|
    (\vlproj{0}\vec{f}-\vec{f},\vec{z}_{\vec{g}}-\vlproj{0}\vec{z}_{\vec{g}})
    \right|
    \lesssim h^2\seminorm[H^1(\Th;\Real^d)]{\vec{f}}\seminorm[H^1(\Omega;\Real^d)]{\vec{z}_{\vec{g}}}
    \lesssim h^2\seminorm[H^1(\Th;\Real^d)]{\vec{f}}\norm{\vec{g}},
  \end{equation}
  where we have used a standard estimate on $\seminorm[H^1(\Omega;\Real^d)]{\vec{z}_{\vec{g}}}$ obtained letting $\vec{v}=\vec{z}_{\vec{g}}$ in \eqref{eq:dual} and using the Cauchy--Schwarz and Korn inequalities to bound the right-hand side.
  On the other hand, using the definitions \eqref{eq:lproj.h} and \eqref{eq:lproj} of the global and local $L^2$-orthogonal projectors, we can write 
  \begin{equation}\label{eq:T2:2}
    \begin{aligned}
      &\left|
      \big(\tens{\sigma}(\GRADs\vec{u}),\GRADs\vec{z}_{\vec{g}}
      \big)
      - \big(
      \tlproj{0}(\tens{\sigma}(\GRADs\vec{u})),\tlproj{0}\GRADs\vec{z}_{\vec{g}}
      \big)
      \right|
      \\
      &\qquad
      =\left|
      \big(
      \tens{\sigma}(\GRADs\vec{u})-\vlproj{0}(\tens{\sigma}(\GRADs\vec{u})),\GRADs\vec{z}_{\vec{g}}-\tlproj{0}(\GRADs\vec{z}_{\vec{g}})
      \big)
      \right|
      \\
      &\qquad
      \le
      \norm{\tens{\sigma}(\GRADs\vec{u})-\tlproj{0}(\tens{\sigma}(\GRADs\vec{u}))}
      ~\norm{\GRADs\vec{z}_{\vec{g}}-\tlproj{0}(\GRADs\vec{z}_{\vec{g}})}
      \\
      &\qquad
      \lesssim h^2\left(
      (2\mu)\seminorm[H^2(\Omega;\Real^d)]{\vec{u}} + \seminorm[H^1(\Omega;\Real^d)]{\lambda\DIV\vec{u}}
      \right)\seminorm[H^2(\Omega;\Real^d)]{\vec{z}_{\vec{g}}}
      \\
      &\qquad
      \lesssim h^2\norm{\vec{f}}~\norm{\vec{g}},
    \end{aligned}
  \end{equation}
  where we have used a Cauchy--Schwarz inequality to pass to the third line,
  \eqref{eq:est.sigma} with $\vec{w}=\vec{u}$ together with the approximation property \eqref{eq:lproj:approx} of the $L^2$-orthogonal projector to pass to the fourth line,
  and the elliptic regularity bound \eqref{eq:dual:a-priori} to conclude.
  Finally, using Cauchy--Schwarz inequalities, we can write
  \begin{equation}\label{eq:T2:3+4}
    \begin{aligned}
      (2\mu)~\mathrm{j}_h(\veproj{1}\vec{u},\veproj{1}\vec{z}_{\vec{g}})
      +(2\mu)~\mathrm{s}_h(\vIh\vec{u},\vIh\vec{z}_{\vec{g}})
      &\lesssim
      (2\mu)~\seminorm[\mathrm{j},h]{\veproj{1}\vec{u}}
      \seminorm[\mathrm{j},h]{\veproj{1}\vec{z}_{\vec{g}}}
      + (2\mu)~\seminorm[\mathrm{s},h]{\vIh\vec{u}}
      \seminorm[\mathrm{s},h]{\vIh\vec{z}_{\vec{g}}}
      \\
      &\lesssim
      (2\mu)h^2\seminorm[H^2(\Omega;\Real^d)]{\vec{u}}
      \seminorm[H^2(\Omega;\Real^d)]{\vec{z}_{\vec{g}}}
      \lesssim
      h^2\norm{\vec{f}}~\norm{\vec{g}},
    \end{aligned}    
  \end{equation}
  where, to pass to the second line, we have used \eqref{eq:sh:consistency} for the terms involving $\mathrm{s}_h$ and we have proceeded as in the estimate of $\term_2$ in the proof of point (ii) of Lemma \ref{lem:ah} for the terms involving $\mathrm{j}_h$ while, to conclude, we have invoked \eqref{eq:dual:a-priori}.
  Taking absolute values in \eqref{eq:l2.err.est:T2:1} and using the estimates \eqref{eq:T2:1}, \eqref{eq:T2:2}, \eqref{eq:T2:3+4} yields $\left|\Cerr{\vec{u}}{\vIh\vec{z}_{\vec{g}}}\right|\lesssim h^2\norm[H^1(\Th;\Real^d)]{\vec{f}}\norm{\vec{g}}$.
  Hence, passing to the supremum, we obtain
  \begin{equation}\label{eq:l2.err.est:T2}
    \term_2
    \lesssim h^2\norm[H^1(\Th;\Real^d)]{\vec{f}}.
  \end{equation}
  Plugging \eqref{eq:l2.err.est:T1} and \eqref{eq:l2.err.est:T2} into \eqref{eq:l2.basic.est} concludes the proof.
\end{proof}


\section{Numerical tests}\label{sec:numerical.tests}

In what follows we verify, through numerical examples, the results stated in the previous section.


\subsection{Two-dimensional quasi-incompressible case}\label{sec:numerical.tests:brenner}

The first test case is inspired by \cite{Brenner:93}: we solve on the unit square $\Omega=(0,1)^2$ the homogeneous Dirichlet problem corresponding to the exact solution such that
$$
\renewcommand{\arraystretch}{1.2}
\vec{u}(\vec{x}) =
\begin{pmatrix}
  (\cos(2\pi x_1)-1)\sin(2\pi x_2)+\frac{1}{1+\lambda}\sin(\pi x_1)\sin(\pi x_2)
  \\
  (1-\cos(2\pi x_2))\sin(2\pi x_1)+\frac{1}{1+\lambda}\sin(\pi x_1)\sin(\pi x_2)
\end{pmatrix}.
$$
The corresponding forcing term is 
$$
\renewcommand{\arraystretch}{1.2}
\vec{f}(\vec{x}) =
\begin{pmatrix}
  -\mu\left[
    4\sin(2\pi x_2)\left(1-2\cos(2\pi x_1)\right)
    -\frac{2}{1+\lambda}\sin(\pi x_1)\sin(\pi x_2)
    \right]
  - \frac{\lambda+\mu}{1+\lambda}\cos(\pi(x_1+x_2))
  \\
  -\mu\left[
    4\sin(2\pi x_1)\left(2\cos(2\pi x_2)-1\right)
    -\frac{2}{1+\lambda}\sin(\pi x_1)\sin(\pi x_2)
    \right]
  - \frac{\lambda+\mu}{1+\lambda}\cos(\pi(x_1+x_2))
\end{pmatrix}.
$$
We take $\mu=1$ and, in order to assess the robustness of the method in the quasi-incompressible limit, we let $\lambda$ vary in $\{1,10^3,10^6\}$.
For the numerical solution, we consider structured and unstructured triangular, Cartesian orthogonal, and deformed quadrangular mesh families; see Figure \ref{fig:numerical.results:brenner:meshes}.
The solutions corresponding to $\lambda=1$ and $\lambda=10^6$ on the finest Cartesian orthogonal mesh are represented in Figure \ref{fig:numerical.results:brenner:solution}, where we have plotted the components of the global displacement reconstruction obtained from the discrete solution according to \eqref{eq:vph}.

The numerical results are collected in Tables \ref{tab:numerical:results:brenner:tria}--\ref{tab:numerical:results:brenner:distQuad}, where the following quantities are monitored:
$\Ndofs$, the number of degrees of freedom;
$\Nnz$, the number of non-zero entries in the problem matrix;
$\tnorm[\mathrm{a},h]{\uvec{u}_h-\vIh\vec{u}}$, the energy-norm of the error;
$\norm{\vec{u}_h-\vlproj{0}\vec{u}}$, the $L^2$-norm of the error estimated in Theorem \ref{thm:l2.err.est}.
Notice that, in view of Remark \ref{rem:err.est:norm.a.h}, in this and in the following numerical tests the energy error is measured using the $\norm[\mathrm{a},h]{{\cdot}}$-norm, whose computation can be done using the already assembled problem matrix.
We additionally display the Estimated Order of Convergence (EOC) which, denoting by $e_i$ an error measure on the $i$th mesh refinement with meshsize $h_i$, is computed as
$$
  {\rm EOC} = \frac{\log e_i - \log e_{i+1}}{\log h_i - \log h_{i+1}}.
$$
In all the cases, the asymptotic EOC match the ones predicted by the theory, that is, $1$ for the energy-norm of the error and $2$ for the $L^2$-norm.
The results additionally highlight the robustness of the method in the quasi-incompressible limit (see Remark \ref{rem:robustness}) and with respect to the mesh, showing errors of comparable magnitude irrespectively of the value of $\lambda$ and of the selected mesh family.

\begin{figure}\centering
  \begin{minipage}{0.35\textwidth}\centering
    \includegraphics[height=4.25cm]{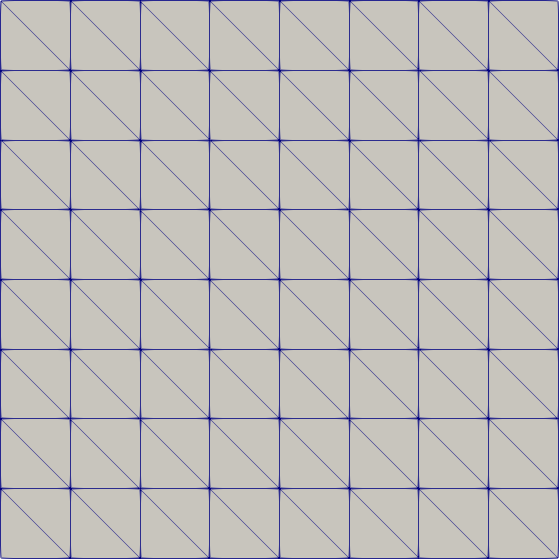}
    \subcaption{Structured triangular mesh}
  \end{minipage}
  \begin{minipage}{0.35\textwidth}\centering
    \includegraphics[height=4.25cm]{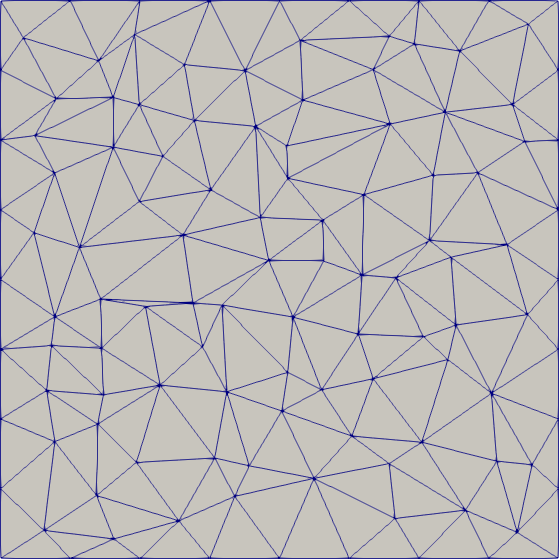}
    \subcaption{Unstructured triangular mesh}
  \end{minipage}
  \vspace{0.25cm} \\
  \begin{minipage}{0.35\textwidth}\centering
    \includegraphics[height=4.25cm]{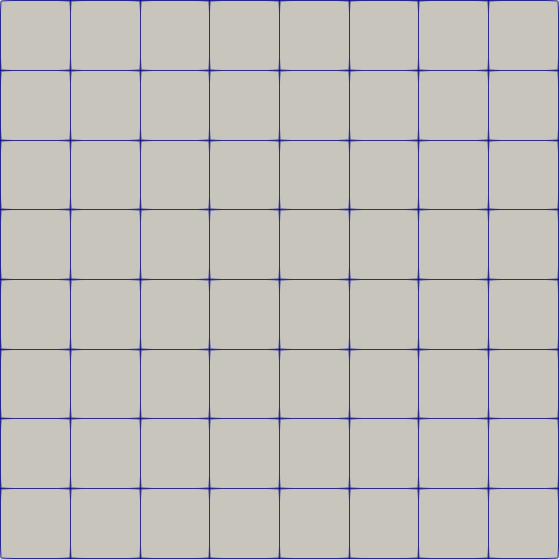}
    \subcaption{Cartesian orthogonal mesh}
  \end{minipage}
  \begin{minipage}{0.35\textwidth}\centering
    \includegraphics[height=4.25cm]{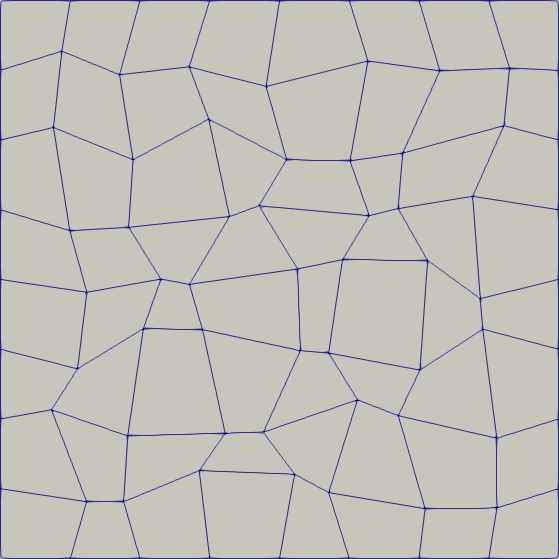}
    \subcaption{Distorted quadrangular mesh}
  \end{minipage}
  \caption{Meshes for the numerical test of Section \ref{sec:numerical.tests:brenner}.\label{fig:numerical.results:brenner:meshes}}
\end{figure}

\begin{figure}\centering
  \begin{minipage}{0.40\textwidth}\centering
    \includegraphics[width=5cm]{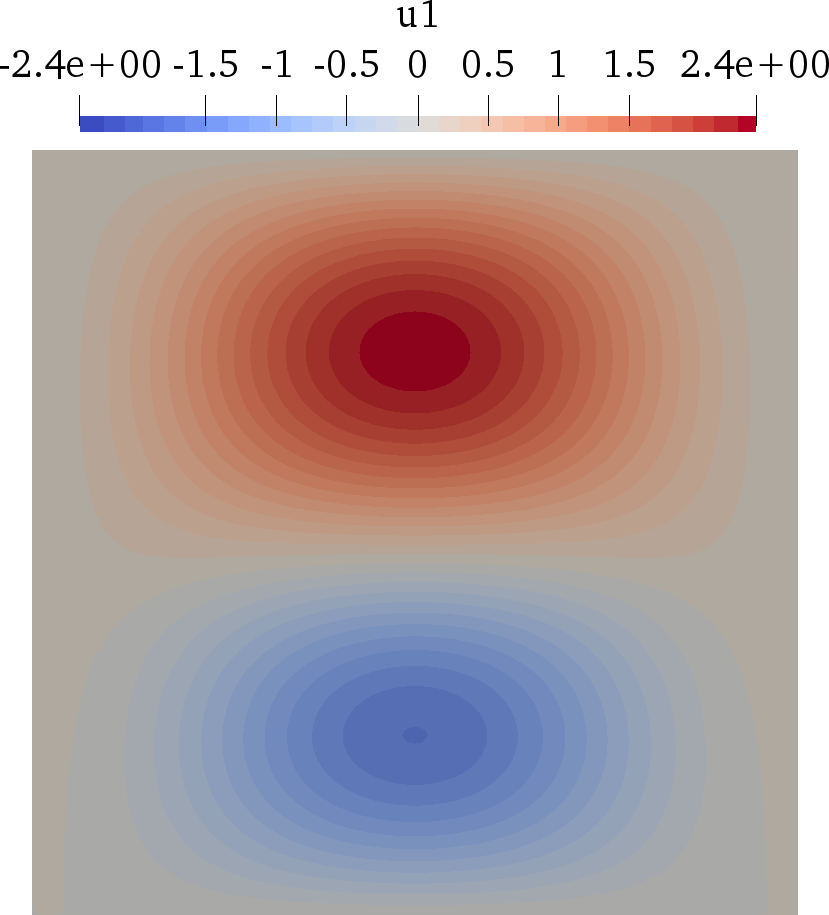}
    \subcaption{$\lambda=1$, $u_1$}
  \end{minipage}
  \begin{minipage}{0.40\textwidth}\centering
    \includegraphics[width=5cm]{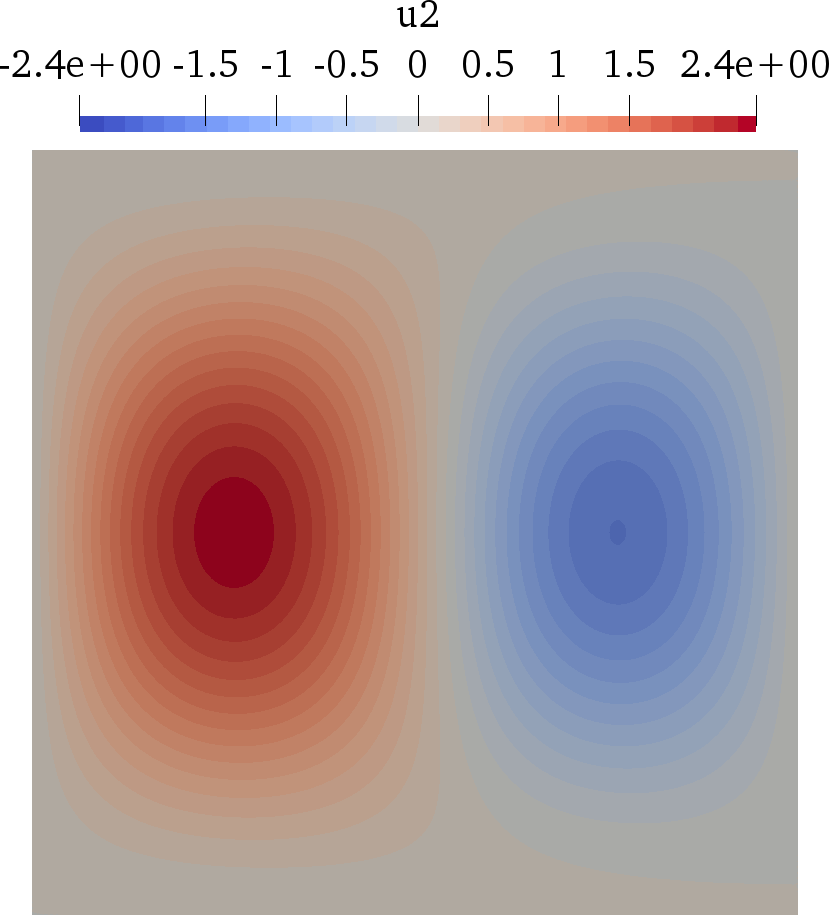}
    \subcaption{$\lambda=1$, $u_2$}
  \end{minipage}
  \vspace{0.25cm} \\
  \begin{minipage}{0.40\textwidth}\centering
    \includegraphics[width=5cm]{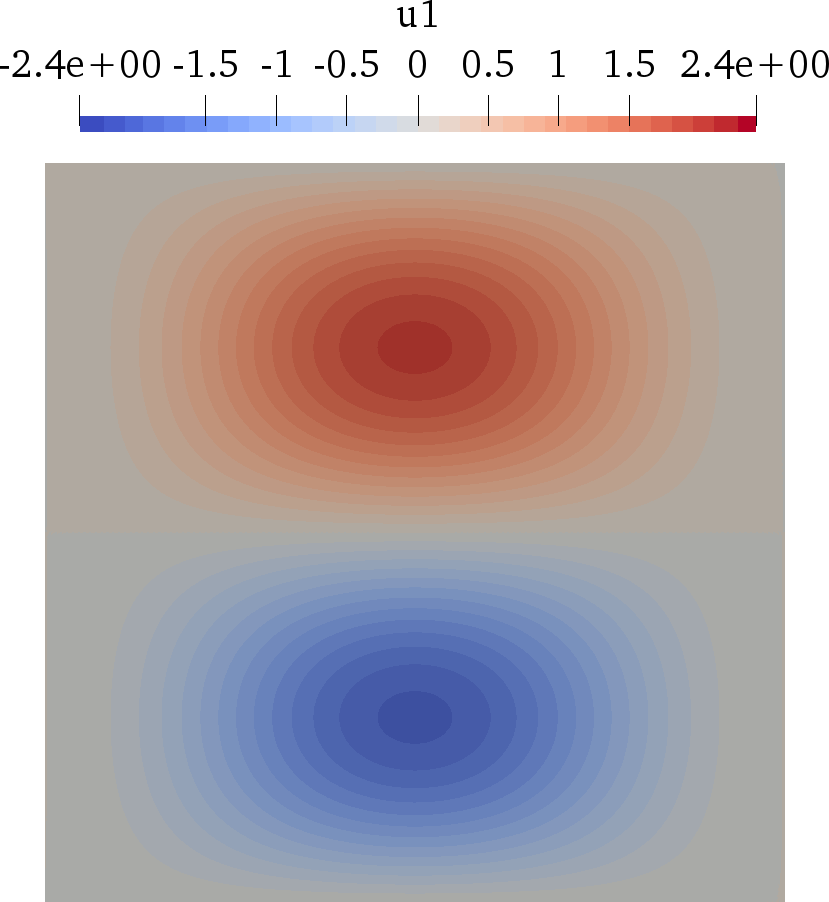}
    \subcaption{$\lambda=\pgfmathprintnumber{1e6}$, $u_1$}
  \end{minipage}
  \begin{minipage}{0.40\textwidth}\centering
    \includegraphics[width=5cm]{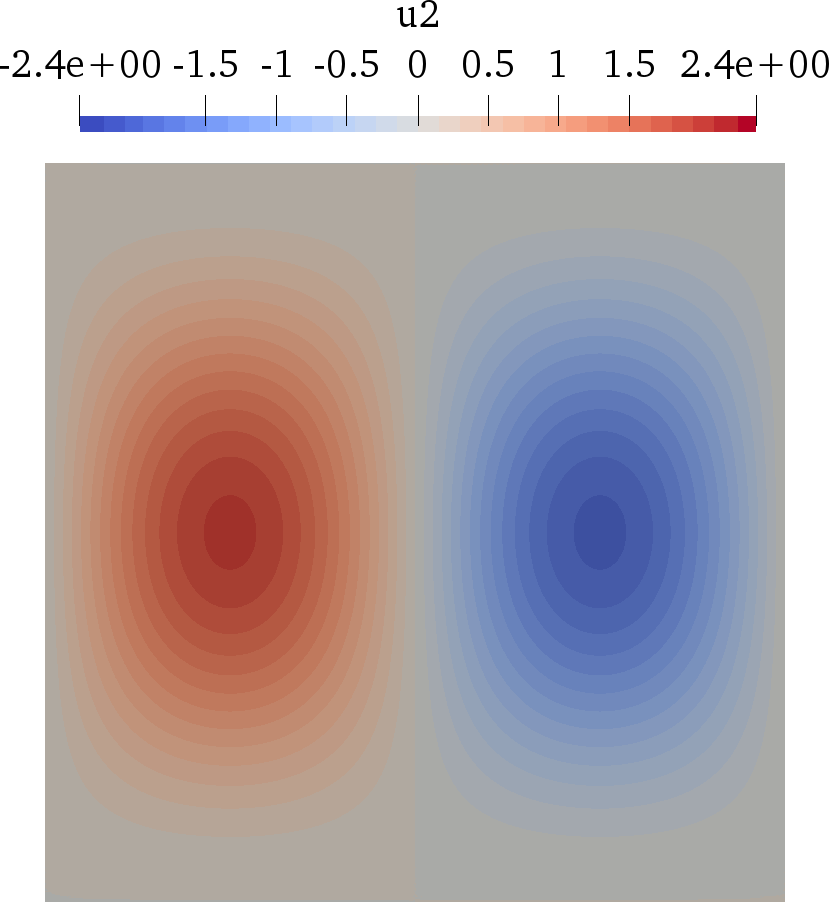}
    \subcaption{$\lambda=\pgfmathprintnumber{1e6}$, $u_2$}
  \end{minipage}
  \caption{Numerical solution for the test of Section \ref{sec:numerical.tests:brenner} on the $128\times128$ Cartesian orthogonal mesh.\label{fig:numerical.results:brenner:solution}}
\end{figure}

\begin{table}\centering
  \caption{Numerical results for the test of Section \ref{sec:numerical.tests:brenner}, structured triangular mesh family.\label{tab:numerical:results:brenner:tria}}
  \small
  \begin{tabular}{cccccc}
    \toprule
    $\Ndofs$
    & $\Nnz$
    & $\norm[\mathrm{a},h]{\uvec{u}_h-\vIh\vec{u}}$
    & EOC
    & $\norm{\vec{u}_h-\vlproj{0}\vec{u}}$
    & EOC   \\
    \midrule
    \multicolumn{6}{c}{$ (\mu,\lambda) = (\pgfmathprintnumber{1.00e+00},\pgfmathprintnumber{1.00e+00}) $ } \\ 
    \midrule
    144        & 3680       & 3.82e+00   & --         & 2.08e-01   & --         \\ 
    608        & 17856      & 1.96e+00   & 0.97       & 6.97e-02   & 1.58       \\ 
    2496       & 78080      & 9.64e-01   & 1.02       & 1.87e-02   & 1.90       \\ 
    10112      & 326016     & 4.84e-01   & 1.00       & 4.74e-03   & 1.98       \\ 
    40704      & 1331840    & 2.43e-01   & 1.00       & 1.19e-03   & 1.99       \\ 
    \midrule
    \multicolumn{6}{c}{$ (\mu,\lambda) = (\pgfmathprintnumber{1.00e+00},\pgfmathprintnumber{1.00e+03}) $ } \\ 
    \midrule
    144        & 3680       & 5.09e+00   & --         & 2.05e-01   & --         \\ 
    608        & 17856      & 1.95e+00   & 1.38       & 7.15e-02   & 1.52       \\ 
    2496       & 78080      & 9.15e-01   & 1.09       & 2.00e-02   & 1.84       \\ 
    10112      & 326016     & 4.52e-01   & 1.02       & 5.18e-03   & 1.95       \\ 
    40704      & 1331840    & 2.25e-01   & 1.00       & 1.31e-03   & 1.98       \\
    \midrule
    \multicolumn{6}{c}{$ (\mu,\lambda) = (\pgfmathprintnumber{1.00e+00},\pgfmathprintnumber{1.00e+06}) $ } \\ 
    \midrule
    144        & 3680       & 1.10e+02   & --         & 2.05e-01   & --         \\ 
    608        & 17856      & 1.48e+01   & 2.90       & 7.15e-02   & 1.52       \\ 
    2496       & 78080      & 2.07e+00   & 2.83       & 2.00e-02   & 1.84       \\ 
    10112      & 326016     & 5.08e-01   & 2.03       & 5.19e-03   & 1.95       \\ 
    40704      & 1331840    & 2.27e-01   & 1.16       & 1.31e-03   & 1.98       \\
    \bottomrule
  \end{tabular}
\end{table}

\begin{table}\centering
  \caption{Numerical results for the test of Section \ref{sec:numerical.tests:brenner}, unstructured triangular mesh family.\label{tab:numerical:results:brenner:distTria}}
  \small
  \begin{tabular}{cccccc}
    \toprule
    $\Ndofs$
    & $\Nnz$
    & $\norm[\mathrm{a},h]{\uvec{u}_h-\vIh\vec{u}}$
    & EOC
    & $\norm{\vec{u}_h-\vlproj{0}\vec{u}}$
    & EOC   \\
    \midrule
    \multicolumn{6}{c}{$ (\mu,\lambda) = (\pgfmathprintnumber{1.00e+00},\pgfmathprintnumber{1.00e+00}) $ } \\ 
    \midrule
    234        & 6572       & 3.00e+00   & --         & 1.38e-01   & --         \\ 
    978        & 30012      & 1.60e+00   & 0.90       & 4.11e-02   & 1.75       \\ 
    3986       & 127372     & 8.15e-01   & 0.98       & 9.37e-03   & 2.13       \\ 
    15542      & 505828     & 4.27e-01   & 0.93       & 2.61e-03   & 1.85       \\ 
    63584      & 2089920    & 2.12e-01   & 1.01       & 6.65e-04   & 1.97       \\ 
    249238     & 8228988    & 1.08e-01   & 0.97       & 1.71e-04   & 1.96       \\
    \midrule
    \multicolumn{6}{c}{$ (\mu,\lambda) = (\pgfmathprintnumber{1.00e+00},\pgfmathprintnumber{1.00e+03}) $ } \\ 
    \midrule
    234        & 6572       & 3.57e+00   & --         & 1.45e-01   & --         \\ 
    978        & 30012      & 1.60e+00   & 1.15       & 4.52e-02   & 1.68       \\ 
    3986       & 127372     & 8.00e-01   & 1.00       & 1.07e-02   & 2.07       \\ 
    15542      & 505828     & 4.18e-01   & 0.94       & 2.99e-03   & 1.85       \\ 
    63584      & 2089920    & 2.08e-01   & 1.01       & 7.63e-04   & 1.97       \\ 
    249238     & 8228988    & 1.06e-01   & 0.97       & 1.97e-04   & 1.96       \\
    \midrule
    \multicolumn{6}{c}{$ (\mu,\lambda) = (\pgfmathprintnumber{1.00e+00},\pgfmathprintnumber{1.00e+06}) $ } \\ 
    \midrule
    234        & 6572       & 6.17e+01   & --         & 1.45e-01   & --         \\ 
    978        & 30012      & 7.55e+00   & 3.03       & 4.52e-02   & 1.68       \\ 
    3986       & 127372     & 1.14e+00   & 2.72       & 1.07e-02   & 2.07       \\ 
    15542      & 505828     & 4.33e-01   & 1.40       & 2.99e-03   & 1.85       \\ 
    63584      & 2089920    & 2.08e-01   & 1.06       & 7.63e-04   & 1.97       \\ 
    249238     & 8228988    & 1.06e-01   & 0.98       & 1.97e-04   & 1.96       \\
    \bottomrule
  \end{tabular}
\end{table}

\begin{table}\centering
  \caption{Numerical results for the test of Section \ref{sec:numerical.tests:brenner}, Cartesian orthogonal mesh family.\label{tab:numerical:results:brenner:quad}}
  \small
  \begin{tabular}{cccccc}
    \toprule
    $\Ndofs$
    & $\Nnz$
    & $\norm[\mathrm{a},h]{\uvec{u}_h-\vIh\vec{u}}$
    & EOC
    & $\norm{\vec{u}_h-\vlproj{0}\vec{u}}$
    & EOC   \\
    \midrule
    \multicolumn{6}{c}{$ (\mu,\lambda) = (\pgfmathprintnumber{1.00e+00},\pgfmathprintnumber{1.00e+00}) $ } \\ 
    \midrule
    80         & 2768       & 3.13e+00   & --         & 1.55e-01   & --         \\ 
    352        & 15856      & 1.84e+00   & 0.77       & 4.08e-02   & 1.93       \\ 
    1472       & 73904      & 1.09e+00   & 0.75       & 1.04e-02   & 1.98       \\ 
    6016       & 317488     & 5.89e-01   & 0.89       & 2.89e-03   & 1.84       \\ 
    24320      & 1314608    & 3.02e-01   & 0.97       & 7.73e-04   & 1.90       \\
    \midrule
    \multicolumn{6}{c}{$ (\mu,\lambda) = (\pgfmathprintnumber{1.00e+00},\pgfmathprintnumber{1.00e+03}) $ } \\ 
    \midrule
    80         & 2768       & 3.08e+00   & --         & 1.64e-01   & --         \\ 
    352        & 15856      & 1.81e+00   & 0.77       & 4.72e-02   & 1.80       \\ 
    1472       & 73904      & 1.08e+00   & 0.75       & 1.37e-02   & 1.78       \\ 
    6016       & 317488     & 5.81e-01   & 0.89       & 3.96e-03   & 1.79       \\ 
    24320      & 1314608    & 2.97e-01   & 0.97       & 1.06e-03   & 1.90       \\
    \midrule
    \multicolumn{6}{c}{$ (\mu,\lambda) = (\pgfmathprintnumber{1.00e+00},\pgfmathprintnumber{1.00e+06}) $ } \\ 
    \midrule
    80         & 2768       & 3.08e+00   & --         & 1.64e-01   & --         \\ 
    352        & 15856      & 1.81e+00   & 0.77       & 4.72e-02   & 1.80       \\ 
    1472       & 73904      & 1.08e+00   & 0.75       & 1.37e-02   & 1.78       \\ 
    6016       & 317488     & 5.81e-01   & 0.89       & 3.96e-03   & 1.79       \\ 
    24320      & 1314608    & 2.97e-01   & 0.97       & 1.06e-03   & 1.90       \\
    \bottomrule
  \end{tabular}
\end{table}

\begin{table}\centering
  \caption{Numerical results for the test of Section \ref{sec:numerical.tests:brenner}, distorted quadrangular mesh family.\label{tab:numerical:results:brenner:distQuad}}
  \small
  \begin{tabular}{cccccc}
    \toprule
    $\Ndofs$
    & $\Nnz$
    & $\norm[\mathrm{a},h]{\uvec{u}_h-\vIh\vec{u}}$
    & EOC
    & $\norm{\vec{u}_h-\vlproj{0}\vec{u}}$
    & EOC   \\
    \midrule
    \multicolumn{6}{c}{$ (\mu,\lambda) = (\pgfmathprintnumber{1.00e+00},\pgfmathprintnumber{1.00e+00}) $ } \\ 
    \midrule
    80         & 2768       & 3.51e+00   & --         & 1.89e-01   & --         \\ 
    352        & 15856      & 1.91e+00   & 0.88       & 5.45e-02   & 1.79       \\ 
    1472       & 73904      & 1.08e+00   & 0.82       & 1.34e-02   & 2.03       \\ 
    6016       & 317488     & 5.83e-01   & 0.89       & 3.52e-03   & 1.93       \\ 
    24320      & 1314608    & 2.97e-01   & 0.97       & 9.18e-04   & 1.94       \\ 
    97792      & 5348656    & 1.49e-01   & 0.99       & 2.33e-04   & 1.98       \\
    \midrule
    \multicolumn{6}{c}{$ (\mu,\lambda) = (\pgfmathprintnumber{1.00e+00},\pgfmathprintnumber{1.00e+03}) $ } \\ 
    \midrule
    80         & 2768       & 3.44e+00   & --         & 1.96e-01   & --         \\ 
    352        & 15856      & 1.87e+00   & 0.88       & 5.89e-02   & 1.73       \\ 
    1472       & 73904      & 1.07e+00   & 0.81       & 1.63e-02   & 1.85       \\ 
    6016       & 317488     & 5.74e-01   & 0.89       & 4.48e-03   & 1.86       \\ 
    24320      & 1314608    & 2.92e-01   & 0.97       & 1.18e-03   & 1.93       \\ 
    97792      & 5348656    & 1.47e-01   & 0.99       & 3.00e-04   & 1.97       \\
    \midrule
    \multicolumn{6}{c}{$ (\mu,\lambda) = (\pgfmathprintnumber{1.00e+00},\pgfmathprintnumber{1.00e+06}) $ } \\ 
    \midrule
    80         & 2768       & 9.12e+00   & --         & 1.96e-01   & --         \\ 
    352        & 15856      & 2.27e+00   & 2.00       & 5.89e-02   & 1.73       \\ 
    1472       & 73904      & 1.08e+00   & 1.07       & 1.63e-02   & 1.85       \\ 
    6016       & 317488     & 5.74e-01   & 0.91       & 4.48e-03   & 1.86       \\ 
    24320      & 1314608    & 2.92e-01   & 0.97       & 1.18e-03   & 1.93       \\ 
    97792      & 5348656    & 1.47e-01   & 0.99       & 3.00e-04   & 1.97       \\
    \bottomrule
  \end{tabular}
\end{table}


\subsection{Two-dimensional singular case}\label{sec:numerical.tests:singular}

We next consider the solution of \cite[Section 5.1]{Ainsworth.Senior:97} which, in polar coordinates $(r,\theta)$, reads
$$
\renewcommand{\arraystretch}{1.2}
\vec{u}(r,\theta) = \begin{pmatrix}
  \frac{1}{2G} r^L\left[ (\kappa - Q(L+1))\cos(L \theta) - L \cos((L-2)\theta)\right]  
  \\
  \frac{1}{2G} r^L\left[ (\kappa + Q(L+1))\sin(L \theta) + L \sin((L-2)\theta)\right]
\end{pmatrix},
$$
where the various parameters take the following numerical values: $\mu=\pgfmathprintnumber{0.65}$, %
$\lambda=\pgfmathprintnumber{0.975}$, %
$G=\frac{5}{13}$, %
$\kappa=\frac{9}{5}$, %
$L=\pgfmathprintnumber[precision=14]{0.5444837367825}$, %
$Q=\pgfmathprintnumber[precision=14]{0.5430755788367}$.
The forcing term in this case is equal to zero, while the Dirichlet boundary condition is inferred from the exact solution.
\begin{figure}\centering
    \begin{tikzpicture}[scale=2]
    \draw[->] (-2.1,0) -- (2.1,0) node[below]{$x$};
    \draw[->] (0,-2.1) -- (0,2.1) node[left]{$y$};
    \draw[very thick] (1.414213562373,0) node[above,anchor=north west] {$(\sqrt{2},0)$} %
    -- (0,1.414213562373) node[right,anchor=south west] {$(0,\sqrt{2})$} %
    -- (-0.707106781186548,0.707106781186548) node[left] {$(-\frac{\sqrt{2}}2,\frac{\sqrt{2}}2)$} %
    -- (0,0) node[right,anchor=north west] {$(0,0)$} %
    -- (-0.707106781186548,-0.707106781186548) node[left] {$(-\frac{\sqrt{2}}2,-\frac{\sqrt{2}}2)$} 
    -- (0,-1.414213562373) node[right,anchor=north west] {$(0,-\sqrt{2})$} %
    -- (1.414213562373,0);
  \end{tikzpicture}
  \caption{Domain for the test case of Section \ref{sec:numerical.tests:singular}.\label{fig:numerical.tests:singular:domain}}
  \begin{minipage}{0.45\textwidth}\centering
    \includegraphics[width=0.85\textwidth]{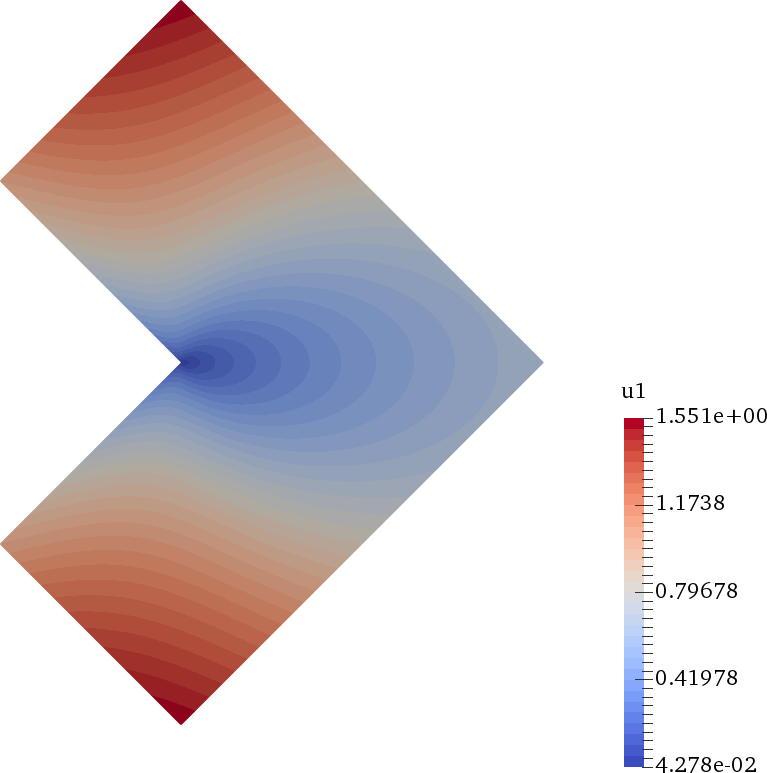}
    \subcaption{$u_1$}
  \end{minipage}
  \begin{minipage}{0.45\textwidth}\centering
    \includegraphics[width=0.85\textwidth]{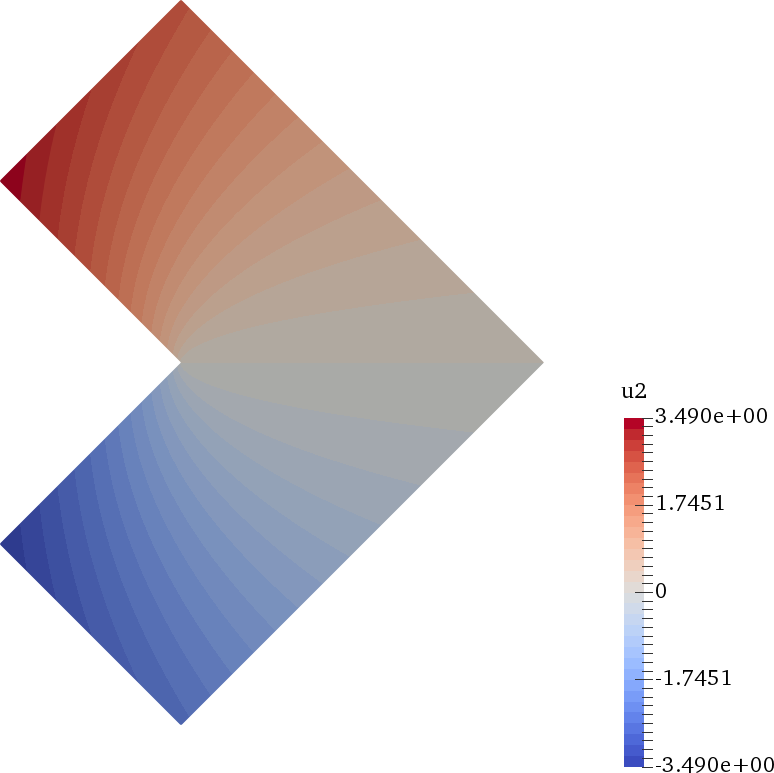}
    \subcaption{$u_2$}
  \end{minipage}
  \caption{Numerical solution for the test of Section \ref{sec:numerical.tests:singular}.\label{fig:numerical.tests:singular:solution}}
\end{figure}
The domain $\Omega$ is illustrated in Figure \ref{fig:numerical.tests:singular:domain}, while the solution on the finest computational mesh considered here is depicted in Figure \ref{fig:numerical.tests:singular:solution}.
This test case is representative of real-life situations corresponding to a mode 1 fracture in a plain strain problem.
The solution exhibits a singularity in the origin, which prevents the method from attaining the full orders of convergence predicted for smooth solutions.

For the numerical resolution, we consider a sequence of refined structured quadrangular meshes.
The numerical results collected in the top half of Table \ref{tab:numerical.tests:singular} show an asymptotic EOC in the energy-norm of about $0.54$, while the asymptotic EOC in the $L^2$-norm is about $1.31$.
For the sake of completeness, we show, in the bottom half of Table \ref{tab:numerical.tests:singular}, a comparison with the original HHO method of \cite{Di-Pietro.Ern:15} with $k=1$.
Also in this case, the EOC are limited by the regularity of the solution, and coincide with those observed for the method studied in this work.
As expected, the number of unknowns on a given mesh is larger for the method of \cite{Di-Pietro.Ern:15} compared to the method proposed here, despite the fact that static condensation is applied in the former case.
It has to be noticed, however, that the reduction in the number of unknowns is balanced by the increased number of nonzero entries in the matrix, due to both the absence of static condensation (see Remark \ref{rem:static.condensation}) and the presence of the jump penalisation term.
This phenomenon is specific to the two-dimensional case: in dimension $d=3$, the matrix corresponding to the method of \cite{Di-Pietro.Ern:15} with $k=1$ is generally more dense; see, e.g., Table \ref{tab:numerical.tests:3d}.
The errors in the energy norm appear to be smaller for the HHO method of \cite{Di-Pietro.Ern:15}, but this is in part due to the fact that the natural energy norm associated with the corresponding bilinear form does not contain the norm of the jumps.
\begin{table}\centering
  \caption{Numerical results for the test of Section \ref{sec:numerical.tests:singular} and comparison with the high-order method of \cite{Di-Pietro.Ern:15} with $k=1$.
    For the latter, the energy norm is the one associated to the corresponding bilinear form without jump stabilisation.\label{tab:numerical.tests:singular}}
  \small
  \begin{tabular}{cccccc}
    \toprule
    $\Ndofs$
    & $\Nnz$
    & $\norm[\mathrm{a},h]{\uvec{u}_h-\vIh\vec{u}}$
    & EOC
    & $\norm{\vec{u}_h-\vlproj{0}\vec{u}}$
    & EOC   \\
    \midrule
    \multicolumn{6}{c}{Present work}- \\ 
    \midrule
    256        & 10616      & 7.65e-01   & --         & 7.51e-02   & --         \\ 
    1088       & 52728      & 5.63e-01   & 0.44       & 3.34e-02   & 1.17       \\ 
    4480       & 232568     & 3.97e-01   & 0.50       & 1.40e-02   & 1.25       \\ 
    18176      & 974712     & 2.76e-01   & 0.53       & 5.72e-03   & 1.29       \\ 
    73216      & 3988856    & 1.90e-01   & 0.54       & 2.31e-03   & 1.31       \\ 
    293888     & 16136568   & 1.31e-01   & 0.54       & 9.29e-04   & 1.31       \\
    \midrule
    \multicolumn{6}{c}{HHO method of \cite{Di-Pietro.Ern:15}, $k=1$} \\
    \midrule
    320        & 7584       & 1.07e-01   & --         & 9.40e-03   & --         \\ 
    1408       & 36512      & 7.32e-02   & 0.55       & 3.64e-03   & 1.37       \\ 
    5888       & 158880     & 5.01e-02   & 0.55       & 1.41e-03   & 1.36       \\ 
    24064      & 661664     & 3.43e-02   & 0.55       & 5.52e-04   & 1.36       \\ 
    97280      & 2699424    & 2.35e-02   & 0.54       & 2.17e-04   & 1.35       \\ 
    391168     & 10903712   & 1.61e-02   & 0.54       & 8.57e-05   & 1.34       \\
    \bottomrule
  \end{tabular}
\end{table}


\subsection{Three-dimensional compressible case}\label{sec:numerical.tests:3d}

To test the performance of the method in three space dimensions, we solve on the unit cube domain $\Omega=(0,1)^3$ the homogeneous Dirichlet problem corresponding to the exact solution $\vec{u}=(u_i)_{1\le i\le d}$ such that
$$
u_i(\vec{x})=\sin(\pi x_1)\sin(\pi x_2)\sin(\pi x_3)\qquad\forall 1\le i\le 3.
$$
The corresponding forcing term is
$$
\renewcommand{\arraystretch}{1.2}
\begin{aligned}
  \vec{f}(\vec{x})
  &=\mu\begin{pmatrix}
  2\sin(\pi x_1)\sin(\pi x_2)\sin(\pi x_3) - \sin(\pi x_2)\cos(\pi(x_3+x_1)) - \sin(\pi x_3)\cos(\pi(x_1+x_2))
  \\
  2\sin(\pi x_1)\sin(\pi x_2)\sin(\pi x_3) - \sin(\pi x_3)\cos(\pi(x_1+x_2)) - \sin(\pi x_1)\cos(\pi(x_2+x_3))
  \\
  2\sin(\pi x_1)\sin(\pi x_2)\sin(\pi x_3) - \sin(\pi x_1)\cos(\pi(x_2+x_3)) - \sin(\pi x_2)\cos(\pi(x_3+x_1))
  \end{pmatrix}  
  \\
  &\quad
  + \lambda\begin{pmatrix}
  \sin(\pi x_1)\sin(\pi x_2)\sin(\pi x_3) - \cos(\pi x_1)\sin(\pi(x_2+x_3))
  \\
  \sin(\pi x_1)\sin(\pi x_2)\sin(\pi x_3) - \cos(\pi x_2)\sin(\pi(x_3+x_1))
  \\
  \sin(\pi x_1)\sin(\pi x_2)\sin(\pi x_3) - \cos(\pi x_3)\sin(\pi(x_1+x_2))
  \end{pmatrix}.
\end{aligned}
$$
For the numerical solution, we take $\mu=\lambda=1$.
Table \ref{tab:numerical.tests:3d} collects the numerical results on Cartesian orthogonal and unstructured simplicial 
mesh families.
The monitored quantities are the same as for the other test cases to which we add, for the sake of comparison, the number of unknowns and of nonzero matrix entries for the method of \cite{Di-Pietro.Ern:15} with $k=1$.
For both mesh families, the asymptotic EOC for both the energy- and the $L^2$-norms of the error agree with the ones predicted.
Specifically, on the simplicial mesh family an EOC close to 1 is attained starting from the third mesh refinement in the energy norm, whereas an EOC close to 2 is already observed starting from the second mesh refinement; on the Cartesian orthogonal mesh family, on the other hand, the orders of convergence take longer to settle to the corresponding asymptotic values, likely because the first computational meshes are very coarse.

\begin{table}\centering
  \caption{Numerical results for the test of Section \ref{sec:numerical.tests:3d}.
    The number of degrees of freedom and of nonzero matrix entries for the method of \cite{Di-Pietro.Ern:15} are also included for comparison (except for the last mesh refinement).
    \label{tab:numerical.tests:3d}}
  \small
  \begin{tabular}{cccccccc}
    \toprule
    \multicolumn{2}{c}{$\Ndofs$}
    & \multicolumn{2}{c}{$\Nnz$}
    & \multirow{2}{*}{$\norm[\mathrm{a},h]{\uvec{u}_h-\vIh\vec{u}}$}
    & \multirow{2}{*}{EOC}
    & \multirow{2}{*}{$\norm{\vec{u}_h-\vlproj{0}\vec{u}}$}
    & \multirow{2}{*}{EOC}
    \\
    $k=0$ & ($k=1$) & $k=0$ & ($k=1$) \\
    \midrule
    \multicolumn{8}{c}{Cartesian orthogonal mesh sequence} \\
    \midrule
    60      & (108)    & 2772      & (4860)     & 2.42e+00   & --         & 1.76e-01   & --         \\ 
    624     & (1296)   & 70128     & (97200)    & 2.07e+00   & 0.23       & 1.01e-01   & 0.81       \\ 
    5568    & (12096)  & 831024    & (1057536)  & 1.31e+00   & 0.65       & 4.09e-02   & 1.30       \\ 
    46848   & (103680) & 7879824   & (9673344)  & 7.19e-01   & 0.87       & 1.27e-02   & 1.68       \\ 
    384000  & (857088) & 68277456  & (82425600) & 3.71e-01   & 0.95       & 3.46e-03   & 1.88       \\
    3108864 & --       & 567808848 & --         & 1.87e-01   & 0.98       & 8.95e-04   & 1.95       \\
    \midrule
    \multicolumn{8}{c}{Unstructured simplicial mesh sequence} \\ 
    \midrule
    1584    & (3024)    & 107136    & (167184)    & 1.38e+00   & --         & 4.70e-02   & --         \\ 
    13248   & (25920)   & 1008288   & (1539648)   & 7.61e-01   & 0.85       & 1.64e-02   & 1.52       \\ 
    108288  & (214272)  & 8676288   & (13125888)  & 3.96e-01   & 0.94       & 4.39e-03   & 1.91       \\ 
    875520  & (1741824) & 71860608  & (108241920) & 2.02e-01   & 0.97       & 1.14e-03   & 1.95       \\
    7041024 & --        & 584706816 & --          & 1.02e-01  & 0.99       & 2.89e-04   & 1.98       \\
    \bottomrule
  \end{tabular}
\end{table}


\section{Local balances and continuity of numerical tractions}\label{sec:flux}

In this section we show that our method satisfies local force balances with equilibrated face tractions.
This property can be exploited, e.g., to derive a posteriori error estimates by flux equilibration, and it makes the proposed method suitable for integration into existing Finite Volume codes.

\begin{lemma}[Traction formulation of the discrete bilinear form]\label{lem:ah:flux}
  We have the following reformulation of the discrete bilinear form $\mathrm{a}_h$ defined by \eqref{eq:ah}:
  For all $\uvec{w}_h,\uvec{v}_h\in\vUhD$,
  \begin{equation}\label{eq:ah:flux}
    \mathrm{a}_h(\uvec{w}_h,\uvec{v}_h)
    = \sum_{T\in\Th}\sum_{F\in\Fh[T]}\meas{F}\vec{\Phi}_{TF}(\uvec{w}_h)\SCAL(\vec{v}_T-\vec{v}_F),
  \end{equation}
  where, for all $T\in\Th$ and all $F\in\Fh[T]$, we have introduced the numerical traction $\vec{\Phi}_{TF}:\vUT\to\Poly{0}(F;\Real^d)$ such that
  $$
  \vec{\Phi}_{TF}(\uvec{w}_h)
  \coloneq -\tens{\sigma}(\GRADs\vpT\uvec{w}_T)\normal_{TF}
  + (2\mu)~\vec{\Phi}_{\mathrm{j},TF}(\uvec{w}_h)
  + (2\mu)~\vec{\Phi}_{\mathrm{s},TF}(\uvec{w}_T),
  $$
  with jump penalisation and stabilisation contributions respectively defined as
  $$
  \begin{aligned}
    \vec{\Phi}_{\mathrm{j},TF}(\uvec{w}_h)
    &\coloneq
    \frac{\epsilon_{TF}}{h_F\meas{F}}\int_F\jump{\vph\uvec{w}_h}
    + \sum_{G\in\Fh[T]}
    \frac{\epsilon_{TG}}{h_G\meas{T}}
    (\overline{\vec{x}}_G-\overline{\vec{x}}_T)\SCAL\normal_{TG}
    \int_G\jump[G]{\vph\uvec{w}_h},
    \\
    \vec{\Phi}_{\mathrm{s},TF}(\uvec{w}_T)
    &\coloneq
    \frac{1}{h_F}\vec{\delta}_{TF}\uvec{w}_T
    + \sum_{G\in\Fh[T]}\frac{\meas{G}}{h_G\meas{T}}(\overline{\vec{x}}_T-\overline{\vec{x}}_G)\SCAL\normal_{TF}~\vec{\delta}_{TG}\uvec{w}_T,
  \end{aligned}
  $$
  where, for any $X$ mesh element or face, we have denoted by $\overline{\vec{x}}_X\coloneq\frac{1}{\meas{X}}\int_X\vec{x}$ its centroid and, for any $T\in\Th$ and any $F\in\Fh[T]$, $\epsilon_{TF}\coloneq\normal_{TF}\SCAL\normal_F$ defines the orientation of $F$ relative to $T$.
\end{lemma}
\begin{proof}
  We proceed to reformulate the three terms in the right-hand side of \eqref{eq:ah} in order to highlight the corresponding contribution to the numerical traction.
  For the consistency term, we can write
  $$
  \begin{aligned}
    (\tens{\sigma}(\GRADsh\vph\uvec{w}_h),\GRADsh\vph\uvec{v}_h)
    &=
    \sum_{T\in\Th}\meas{T}\tens{\sigma}(\GRADs\vpT\uvec{w}_T)\SSCAL\GRADs\vpT\uvec{v}_T  
    =
    \sum_{T\in\Th}\meas{T}\tens{\sigma}(\GRADs\vpT\uvec{w}_T)\SSCAL\GRAD\vpT\uvec{v}_T
    \\
    &=
    \sum_{T\in\Th}\meas{T}\tens{\sigma}(\GRADs\vpT\uvec{w}_T)\SSCAL\left(
    \sum_{F\in\Fh[T]}\frac{\meas{F}}{\meas{T}}(\vec{v}_F-\vec{v}_T)\otimes\normal_{TF}
    \right)
    \\
    &=
    -\sum_{T\in\Th}\sum_{F\in\Fh[T]}
    \meas{F}\tens{\sigma}(\GRADs\vpT\uvec{w}_T)\normal_{TF}\SCAL (\vec{v}_T-\vec{v}_F),
  \end{aligned}
  $$
  where we have used the fact that, for any $T\in\Th$, both $\tens{\sigma}(\GRADs\vpT\uvec{w}_T)$ and $\GRADs\vpT\uvec{v}_T$ are constant inside $T$ along with the fact that $\tens{\sigma}(\GRADs\vpT\uvec{w}_T)$ is symmetric to replace $\GRADs$ with $\GRAD$ in the first line,
  the first relation in \eqref{eq:vpT} to pass to the second line,
  and we have rearranged the products and sums to conclude.

  For the jump penalisation term, we can start by observing that
  $$
  \begin{aligned}
    \mathrm{j}_h(\uvec{w}_h,\uvec{v}_h)
    &=
    \sum_{F\in\Fh}\frac{1}{h_F}\left(\jump{\vph\uvec{w}_h},\jump{\vph\uvec{v}_h}\right)_F
    \\
    &=
    \sum_{F\in\Fh}\sum_{T\in\Th[F]}\frac{\epsilon_{TF}}{h_F}\left(\jump{\vph\uvec{w}_h},\vpT\uvec{v}_T\right)_F
    =
    \sum_{T\in\Th}\sum_{F\in\Fh[T]}\frac{\epsilon_{TF}}{h_F}\left(\jump{\vph\uvec{w}_h},\vpT\uvec{v}_T\right)_F,
  \end{aligned}
  $$
  where we have used the definition of the jump operator to pass to the second line and exchanged the sums over elements and faces according to \eqref{eq:sumTF=sumFT} to conclude.
  Using the explicit expression \eqref{eq:vpT:bis} of the local displacement reconstruction, we can go on writing
  $$
  \begin{aligned}
    \mathrm{j}_h(\uvec{w}_h,\uvec{v}_h)
    &=
    \sum_{T\in\Th}\sum_{F\in\Fh[T]}\frac{\epsilon_{TF}}{h_F}\left(
    \jump{\vph\uvec{w}_h},\vec{v}_T + \sum_{G\in\Fh[T]}\frac{\meas{G}}{\meas{T}}(\vec{x}-\overline{\vec{x}}_T)\SCAL\normal_{TF}(\vec{v}_G-\vec{v}_T)
    \right)_F
    \\
    &=
    \sum_{T\in\Th}\sum_{F\in\Fh[T]}\frac{\epsilon_{TF}}{h_F}\left(\jump{\vph\uvec{w}_h},\vec{v}_T-\vec{v}_F\right)_F
    \\
    &\qquad
    + \sum_{T\in\Th}\sum_{F\in\Fh[T]}\sum_{G\in\Fh[T]}\frac{\epsilon_{TF}\meas{G}}{h_F\meas{T}}\left(
    \jump{\vph\uvec{w}_h}(\vec{x}-\overline{\vec{x}}_T)\SCAL\normal_{TF},\vec{v}_G-\vec{v}_T
    \right)_F
    \\
    &=
    \sum_{T\in\Th}\sum_{F\in\Fh[T]}\meas{F}\left(\frac{\epsilon_{TF}}{h_F\meas{F}}\int_F\jump{\vph\uvec{w}_h}\right)\SCAL(\vec{v}_T-\vec{v}_F)
    \\
    &\qquad
    - \sum_{T\in\Th}\sum_{G\in\Fh[T]}\meas{G}\left(
    \sum_{F\in\Fh[T]}\frac{\epsilon_{TF}}{h_F\meas{T}}\int_F\jump{\vph\uvec{w}_h}(\vec{x}-\overline{\vec{x}}_T)\SCAL\normal_{TF}
    \right)\SCAL(\vec{v}_T-\vec{v}_G)
    \\
    &=
    \sum_{T\in\Th}\sum_{F\in\Fh[T]}\meas{F}\vec{\Phi}_{\mathrm{j},TF}(\uvec{w}_h)\SCAL(\vec{v}_T-\vec{v}_F),
  \end{aligned}
  $$
  where, to insert $\vec{v}_F$ into the first term in the second line, we have used the fact that $\jump{\vph\uvec{w}_h}$ is single-valued at interfaces together with $\vec{v}_F=\vec{0}$ on boundary faces,
  to pass to the third line we have used the fact that the discrete unknowns in $\uvec{v}_h$ are constant over mesh elements to take them out of the integrals over faces
  while, to conclude, we have observed that $(\vec{x}-\overline{\vec{x}}_T)\SCAL\normal_{TF}=(\overline{\vec{x}}_F-\overline{\vec{x}}_T)\SCAL\normal_{TF}$ for all $\vec{x}\in F$ and we have used the definition of $\vec{\Phi}_{\mathrm{j},TF}(\uvec{w}_h)$ after switching the names of the mute variables $F$ and $G$ in the second term of the third line.

  Moving to the stabilisation term, we can write
  $$
  \begin{aligned}
    \mathrm{s}_h(\uvec{w}_h,\uvec{v}_h)
    &=
    \sum_{T\in\Th}\sum_{F\in\Fh[T]}\frac{\meas{F}}{h_F}\vec{\delta}_{TF}\uvec{w}_T\SCAL\vec{\delta}_{TF}\uvec{v}_T
    \\
    &=
    \sum_{T\in\Th}\sum_{F\in\Fh[T]}\frac{\meas{F}}{h_F}\vec{\delta}_{TF}\uvec{w}_T\SCAL\left(
    \vec{v}_T - \vec{v}_F + \sum_{G\in\Fh[T]}\frac{\meas{G}}{\meas{T}}(\vec{v}_G-\vec{v}_T)(\overline{\vec{x}}_F-\overline{\vec{x}}_T)\SCAL\normal_{TG}
    \right)
    \\
    &=
    \sum_{T\in\Th}\sum_{F\in\Fh[T]}\frac{\meas{F}}{h_F}\vec{\delta}_{TF}\uvec{w}_T\SCAL(\vec{v}_T - \vec{v}_F)
    \\
    &\qquad
    + \sum_{T\in\Th}\sum_{G\in\Fh[T]}\meas{G}\left(
    \sum_{F\in\Fh[T]}\frac{\meas{F}}{h_F\meas{T}}(\overline{\vec{x}}_T-\overline{\vec{x}}_F)\SCAL\normal_{TG}~\vec{\delta}_{TF}\uvec{w}_T
    \right)\SCAL(\vec{v}_T-\vec{v}_G)
    \\
    &=
    \sum_{T\in\Th}\sum_{F\in\Fh[T]}\meas{F}\vec{\Phi}_{\mathrm{s},TF}(\uvec{w}_T)
    \SCAL(\vec{v}_T-\vec{v}_F),
  \end{aligned}
  $$
  where we have used the definition \eqref{eq:dTF} of the boundary difference operator together with the explicit expression \eqref{eq:vpT:bis} of the local displacement reconstruction to pass to the second line,
  we have rearranged the terms to pass to the third line,
  and we have used the definition of $\vec{\Phi}_{\mathrm{s},TF}(\uvec{w}_T)$ after switching the names of the mute variables $F$ and $G$ in the second term of the third line to conclude.
\end{proof}

\begin{corollary}[Local balances and equilibrated tractions]
  Under the assumptions and notations of Lemma \ref{lem:ah:flux}, we have that $\uvec{u}_h\in\vUhD$ solves the discrete problem \eqref{eq:discrete} if and only if:
  For all $T\in\Th$ the following balance holds
  \begin{equation}\label{eq:balance}
    \sum_{F\in\Fh[T]}\meas{F}\vec{\Phi}_{TF}(\uvec{u}_h) = \int_T\vec{f},
  \end{equation}
  and, for any interface $F\in\Fhi$ shared by the mesh elements $T_1$ and $T_2$, it holds that
  \begin{equation}\label{eq:equilibrated.tractions}
    \vec{\Phi}_{T_1F}(\uvec{u}_h) + \vec{\Phi}_{T_2F}(\uvec{u}_h) = \vec{0}.
  \end{equation}
\end{corollary}
\begin{proof}
  Plugging the flux reformulation \eqref{eq:ah:flux} of the bilinear form $\mathrm{a}_h$ into the discrete problem \eqref{eq:discrete}, and recalling \eqref{eq:vh}, we infer that it is equivalent to:
  Find $\uvec{u}_h\in\vUhD$ such that
  \begin{equation}\label{eq:discrete:flux}
    \sum_{T\in\Th}\sum_{F\in\Fh[T]}\meas{F}\vec{\Phi}_{TF}(\uvec{u}_h)\SCAL(\vec{v}_T-\vec{v}_F)
    = \sum_{T\in\Th}\int_T\vec{f}\SCAL\vec{v}_T
    \qquad\forall\uvec{v}_h\in\vUhD.
  \end{equation}
  Taking, for a given mesh element $T\in\Th$, $\uvec{v}_h$ such that $\vec{v}_{T'}=\vec{0}$ for all $T'\in\Th\setminus\{T\}$, $\vec{v}_F=\vec{0}$ for all $F\in\Fh$, and letting $\vec{v}_T$ span $\Poly{0}(T;\Real^d)$, \eqref{eq:discrete:flux} reduces to \eqref{eq:balance}.
  Similarly, given an interface $F\in\Fhi$ shared by the mesh elements $T_1$ and $T_2$, taking in \eqref{eq:discrete:flux} $\uvec{v}_h$ such that $\vec{v}_T=\vec{0}$ for all $T\in\Th$, $\vec{v}_{F'}=\vec{0}$ for all $F'\in\Fh\setminus\{F\}$, and letting $\vec{v}_F$ span $\Poly{0}(F;\Real^d)$, \eqref{eq:discrete:flux} reduces to \eqref{eq:equilibrated.tractions} after recalling that the numerical tractions are constant over $F$.
\end{proof}


\section*{Acknowledgements}
The work of the second author was partially supported by \emph{Agence Nationale de la Recherche} grants HHOMM (ANR-15-CE40-0005) and fast4hho (ANR-17-CE23-0019).


\bibliographystyle{plain}
\bibliography{ello}

\end{document}